\documentclass[12pt]{article}
\usepackage{amsmath}
\usepackage{amsfonts}
\usepackage{mitpress}
\usepackage{verbatim}
\usepackage{graphicx}
\usepackage{caption}
\usepackage[bookmarks=true,         bookmarksnumbered=true,
colorlinks=true,
pdfstartview=FitV,
linkcolor=blue, citecolor=blue, urlcolor=blue,
]{hyperref}
\usepackage[sort,compress]{cite}

\newtheorem{theorem}{Theorem}

\newtheorem{algorithm}[theorem]{Algorithm}

\newtheorem{definition}[theorem]{Definition}
\newtheorem{example}[theorem]{Example}

\newtheorem{lemma}[theorem]{Lemma}

\newtheorem{proposition}[theorem]{Proposition}
\newtheorem{remark}[theorem]{Remark}

\numberwithin{equation}{section}
\numberwithin{theorem}{section}
\newenvironment{proof}[1][Proof]{\noindent\textbf{#1.} }{ \rule{0.5em}{0.5em}}
\newcommand{\Ito}{It\^o}
\usepackage{soul}
\IfFileExists{kfont-concrete.sty}{\usepackage{kfont-concrete}}{}
\newcounter{kfoo}

\begin{document}

\title{ Comparison of continuous and discrete-time data-based modeling for hypoelliptic systems }

\author{Fei Lu\thanks{Department of Mathematics, University of
    California, Berkeley and Lawrence Berkeley National
    Laboratory. E-mail addresses: feilu@berkeley.edu (FL);
    chorin@math.berkeley.edu (AC)}, Kevin K.~Lin\thanks{School of
    Mathematical Sciences, University of Arizona.  E-mail address:
    klin@math.arizona.edu}, and Alexandre J.~Chorin$^*$}

\maketitle

\begin{abstract}
  We compare two approaches to the predictive modeling of dynamical
  systems from partial observations at discrete times.  The first is
  continuous in time, where one uses data to infer a model in the form
  of stochastic differential equations, which are then discretized for
  numerical solution. The second is discrete in time, 
  where one directly infers a discrete-time model in the form of a nonlinear autoregression moving average model.  The comparison is performed in a
  special case where the observations are known to have been obtained
  from a hypoelliptic stochastic differential equation.  We show that the
  discrete-time approach has better predictive skills, especially when
  the data are relatively sparse in time.  We discuss open questions as
  well as the broader significance of the results.
  
\end{abstract}


\section{Introduction}

We examine the problem of inferring predictive stochastic models for a
dynamical system, given partial observations of the system at a discrete sequence of times. This inference problem arises in
applications ranging from molecular dynamics to climate modeling (see,
e.g. \cite{FS01, GCF15} and references therein).  The observations may
come from a stochastic or a deterministic chaotic system.  This
inference process, often called stochastic parametrization, is useful
both for reducing computational cost by constructing effective
lower-dimensional models, and for making prediction possible when
fully-resolved measurements of initial data and/or a full model are not
available.

Typical approaches to stochastic parametrization begin by identifying a
continuous-time model, usually in the form of stochastic differential
equations (SDEs), then discretizing the resulting model to make
predictions.  One difficulty with this standard approach is that it
often leads to hypoelliptic systems~\cite{KCG15,
  MH13, PSW09}, in which the noise acts on a proper
subset of state space directions.  As we will explain, this degeneracy
can make parameter estimation for hypoelliptic systems particularly difficult~\cite{PSW09, Sor12, ST12}, makiing the resulting model a
poor predictor for the system at hand.

Recent work~\cite{CL15,LLC15} has shown that fully discrete-time approaches to stochastic
parametrization, in which one considers a discrete-time parametric model
and infers its parameters from data, have certain advantages over
continuous-time methods.  In this paper, we compare the standard,
continuous-time approach with a fully discrete-time approach, in a  
 special case where the observations are
known in advance to have been produced by a hypoelliptic system whose
form is known, and only some parameters remain to be inferred.  
We hope that this comparison, in a relatively simple and 
well-understood context, will clarify some of the advantages and
disadvantages of discrete-time modeling for dynamical systems.
We note
that our discussion here leaves in abeyance the question of what to do in cases
where much less is known about the origin of the data; in general, there
is no reason to believe that a given set of observations was generated
by any stochastic differential equation or by a Markovian model of any kind. 

A major difficulty in discrete modeling is the derivation of the structure, i.e. of the terms in the discrete-time model. We show that when the form of the differential equation giving rise to the data is known,
one can deduce possible terms for the discrete model, but not
necessarily the associated coefficients, from numerical schemes. Note
that the use of this idea places the discrete and continuous
models we compare on an equal footing, in that both approaches produce
models directly derived from the assumed form of the model.

%\fnote{They are not on an equal footing. This is CT's "playground".}

\paragraph{Model and goals.}
The specific hypoelliptic stochastic differential equations we work with
have the form
\begin{eqnarray}
  dx_{t} &=&y_{t}~dt,  \label{Keq} \\
  dy_{t} &=&\big(-\gamma y_{t}-V^{\prime }(x_{t})\big)~dt+\sigma dB_{t}~,  \notag
\end{eqnarray}
where $B_{t}$ is a standard Wiener process.  When the potential $V$ is
quadratic, i.e.,
\begin{equation*}
  V(x)=\frac{\alpha }{2}x^{2}~,~~\alpha >0,
\end{equation*}
we get a linear Langevin equation.  When the potential has the form
\begin{equation*}
  V(x)=\frac{\beta }{4}x^{4}-\frac{\alpha }{2}x^{2}~,~~\alpha,\beta >0,
\end{equation*}
this is the Kramers oscillator~\cite{Kra40, SGH93, AI00,hum05}.  It
describes the motion of a particle in a double-well potential driven by
white noise, with $x_{t}$ and $y_{t}$ being the position and the
velocity of the particle; $\gamma>0$ is a damping constant. The white
noise represents the thermal fluctuations of a surrounding ``heat
bath'', the temperature of which is connected to $\gamma$ and $\sigma$
via the Einstein relation $T=\frac{\sigma^2}{2\gamma}$.  This system is
ergodic, with stationary density $p(x,y)\propto\exp \left(
-\frac{2\gamma }{\sigma ^{2}}\left( \frac{1}{2} y^{2}+V(x)\right)
\right)$.  It has multiple time scales and can be highly nonlinear, but
is simple enough to permit detailed numerical study. Parameter
estimation for this system is also rather well-studied~\cite{PSW09,
  ST12}.  These properties make Eq.~(\ref{Keq}) a natural example
for this paper.

One of our goals is to construct a model that can make short-time
forecasts of the evolution of the variable $x$ based on past
observations $\left\{ {x}_{nh}\right\} _{n=1}^{N}$, where $h>0$ is the
observation spacing, in the situation where the parameters $\gamma$,
$\alpha$, $\beta$, and $\sigma$ are unknown.  (The variable $y$ is not
observed, hence even when the parameters are known, the initial value of
$y$ is missing when one tries to solve the SDEs to make predictions.)
We also require that the constructed model be able to reproduce
long-term statistics of the data, e.g., marginals of the stationary
distribution.  In part, this is because the form of the model (either
continuous or discrete-time) is generally unknown, and
reproduction of long-term statistics provides a useful criterion for
selecting a particular model.  But even more important, in order for a
model to be useful for tasks like data assimilation and uncertainty
quantification, it must faithfully capture relevant statistics on time
scales ranging from the short term (on which
trajectory-wise forecasting is possible) to longer time scales.

Our main finding is that the discrete-time approach makes predictions as
reliably as the true system that gave rise to the data (which is
of course unknown in general), even for relatively large observation spacings,  while a continuous-time approach is only accurate when the observation
spacing $h$ is small, even in very low-dimensional examples such as ours.

\paragraph{Paper organization.}
We briefly review some basic facts about hypoelliptic systems in
Section~\ref{sect:bkgnd}, including the parameter estimation technique
we use to implement the continuous-time approach.  In
Section~\ref{sec3}, we discuss the discrete-time approach.
Section~\ref{sec4} presents numerical results, and in Section~\ref{sec5}
we summarize our findings and discuss broader implications of our
results.  For the convenience of the reader, we collect a number of
standard results about SDEs and their numerical solutions in the
Appendices.

%%%%%%%%%%%%%%%%%%%%%%%%%%%%%%%%
% \section{Brief review of hypoelliptic systems}
\section{Brief review of the continuous-time approach}\label{sect:bkgnd}

\subsection{Inference for partially observed hypoelliptic systems}

Consider a stochastic differential equation of the form
\begin{equation}
  \label{eq2.1}
  \begin{array}{rcl}
    dX&=&f(X,Y)~dt \\
    dY &= &a(X,Y)~dt + b(X,Y)~dW_t~.\\
  \end{array}
\end{equation}
Observe that only the $Y$ equation is stochastically forced. Because of
  this, the second-order operator in the Fokker-Planck equation
\begin{equation}
  \label{eq:FP}
  \frac{\partial}{\partial t}p(x,y,t) = - \frac{\partial}{\partial x} [f(x,y)p(x,y,t)] - \frac{\partial}{\partial y} [a(x,y)p(x,y,t)] 
  + \frac{1}{2}\frac{\partial^2}{\partial y^2} [b^2(x,y)p(x,y,t)] 
\end{equation}
for the time evolution of probability densities is not elliptic.  This
means that without any further assumptions on Eq.~(\ref{eq2.1}), the
solutions of the Fokker-Planck equation, and hence the transition
probability associated with the SDE, might be singular in the $X$
direction.  Hypoellipticity is a condition that guarantees the
existence of smooth solutions for Eq.~(\ref{eq:FP}) despite this
degeneracy.  Roughly speaking, a system is hypoelliptic if the drift
terms (i.e., the vector fields $f(x,y)$ and $a(x,y)$) help to spread the
noise to all phase space directions, so that the system has a
nondegenerate transition density. Technically, hypoellipticity requires certain conditions involving the
  Lie brackets of drift and diffusion fields, known as H\"ormander's
  conditions \cite{Nua06}; when these conditions are satisfied, the
  system can be shown to possess smooth transition densities.

Our interest is in systems for which 
only discrete observations of $x$ are
available, and we use these observations to estimate the
parameters in the functions $f, a,b$.  While parameter estimation for
completely observed nondegenerate systems has been widely investigated (see
e.g. \cite{Rao99, Sor12}), and there has been recent progress
  toward parameter estimation for partially-observed nondegenerate
  systems~\cite{Jen14}, parameter estimation from discrete partial observations for hypoelliptic systems
 remains challenging.

There are three main categories of methods for parameter
estimation (see, e.g., the surveys \cite{Sor04}, and \cite{Sor12}):
\begin{enumerate}
\item Likelihood-type methods, where the likelihood is analytically or
  numerically approximated, or a likelihood-type function is constructed
  based on approximate equations.  These methods lead to maximum
  likelihood estimators (MLE).

\item Bayesian methods, in which one combines a prior with a likelihood,
  and one uses the posterior mean as estimator. Bayesian methods are
  important when the likelihood has multiple maxima.  However, suitable priors may not always be
  available.
  
\item Estimating function methods, or generalized moments methods, where
  estimators are found by estimating functions of parameters and
  observations.  These methods generalize likelihood-type methods,
    and are useful when transition densities (and hence likelihoods) are
 difficult to compute. Estimating functions can be constructed using associated martingales or moments.

\end{enumerate}
Because projections of Markov processes are typically not Markov, and
the system is hypoelliptic, all three of the
  above approaches face difficulties for systems like (\ref{Keq}): the
likelihood function is difficult to compute either analytically or numerically, because only partial observations are available, and likelihood-type
functions based on approximate equations often lead to biased
estimators\cite{Glo06,PSW09,ST12}.  There are also no easily
calculated martingales on which to base estimating
functions \cite{DS04}.

There are two special cases that have been well-studied.  When the
system is linear, the observed process is a continuous-time
autoregression process. Parameter estimation for this case is
well-understood, see, e.g., the review papers \cite{Bro01,BDY07}. When
the observations constitute an integrated diffusion (that is, $f(x,y)=y$
and the $Y$ equation is autonomous, so that $X$ is an integral of the
diffusion process $Y$), consistent, asymptotically normal estimators are
constructed in \cite{DS04} using prediction-based estimating functions,
and in \cite{Glo06} using a likelihood type method based on Euler
approximation. However, these approaches rely on the system being linear
or the unobserved process being autonomous, and are not adapted to
general hypoelliptic systems.

To our knowledge, for general hypoelliptic systems with discrete partial
observation, only Bayesian type methods \cite{PSW09} and a likelihood
type method {\cite{ST12}} have been proposed when $f(x,y)$ is such that
Eq.~(\ref{eq2.1}) can be written in the form of Eq.~(\ref{Keq}) by a
change of variables.  In \cite{PSW09} Euler and \Ito-Taylor
approximations are combined in a deterministic scan Gibbs sampler
alternating between parameters and missing data in the unobserved
variables.  The reason for combining Euler and \Ito-Taylor approximation
is that Euler approximation leads to underestimated MLE of diffusion but
is effective for drift estimation, whereas \Ito-Taylor expansion leads
to unbiased MLE of diffusion but is inappropriate for drift
estimation. In \cite{ST12} explicit consistent maximum likelihood-type
estimators are constructed. However, all these methods require the observation
spacing $h$ to be small and the number of observations $N$ to be large.
For example, the estimators in \cite{ST12} are only guaranteed to
converge if, as $N\to\infty$, $h\to0$ in
such a way that $Nh^{2}\rightarrow 0$ and $Nh\rightarrow \infty$.  In
practice, the observation spacing $h>0$ is fixed, and large biases have
been observed when $h$ is not sufficiently small \cite{PSW09,ST12}.  We
show in this paper that the bias can be so large that the prediction
from the estimated system may be unreliable.

% \section{Brief review of hypoelliptic systems}

\subsection{Continuous-time stochastic parametrization}
\label{sect:CT stochastic parametrization}

The continuous-time approach starts by proposing a parametric
  hypoelliptic system and estimating parameters in the system from
\emph{discrete partial observations}. In the present paper, the
form of the hypoelliptic system is assumed to be known. Based on the Euler scheme approximation of the second equation in the
system, Samson and Thieullen \cite{ST12} constructed the following
likelihood-type function, or ``contrast''
\begin{equation*}
L_{N}(\theta )=\sum_{n=1}^{N-3}\frac{3}{2}\frac{\left[ \hat{y}_{(n+2)h}-\hat{y}_{\left( n+1\right) h}+h(\gamma \hat{y}_{nh}+V^{\prime }(x_{nh}))\right]
^{2}}{h\sigma ^{2}}+\left( N-3\right) \log \sigma ^{2},
\end{equation*}%
where $\theta =\left( \gamma ,\beta ,\alpha ,\sigma ^{2}\right) $ and 
\begin{equation}
  \hat{y}_{n}=\frac{x_{\left( n+1\right) h}-x_{nh}}h.  \label{xBar}
\end{equation}%
Note that a shift in time in the drift term, i.e. the time index of
$\gamma\hat{y}_{nh}+V^{\prime }(x_{nh})$ is $nh$ instead of $(n+1)h$, is
introduced to avoid a $\sqrt{h}$ correlation between
$\hat{y}_{(n+2)h}-\hat{y}_{\left( n+1\right) h}$ and
$\gamma\hat{y}_{(n+1)h}+V^{\prime }(x_{(n+1)h})$.  Note also that there
is a weighting factor $\frac{3}{2}$ in the sum, because the maximum
likelihood estimator based on Euler approximation underestimates the
variance (see, e.g., \cite{Glo06,PSW09}).

 The estimator is the minimizer of the contrast
\begin{equation}
\hat{\theta}_{N}=\arg \min_{\theta }L_{N}(\theta ).  \label{MLE}
\end{equation}
The estimator $\hat{\theta}_{N}$ converges to the true parameter value
$\theta =\left( \gamma ,\beta ,\alpha ,\sigma ^{2}\right) $ under the
condition that $h\rightarrow 0$, $Nh\rightarrow \infty $ and
$Nh^{2}\rightarrow 0$. However, if $h$ is not small enough, the
estimator $\hat{\theta}_{N}$ can have a large bias (see in \cite{ST12}
and in the later sections), and the bias can be so large that the
estimated system may have dynamics very different from the true system, 
and its prediction becomes unreliable.

\vspace{2mm}
\noindent{\bf Remark 2.1} 
In the case $V'(x)=\alpha x$, the Langevin system $(\ref{Keq})$ is
linear. The process $\{x_t,t\geq 0\}$ is a continuous-time
autoregressive process of order two, and there are
various ways to estimate the parameters {\rm(}see the review {\rm
  \cite{Bro14})}, e.g., the likelihood method using a state-space
representation and a Kalman recursion {\rm\cite{Jon81}}, or methods for
fitting discrete-time ARMA models {\rm\cite{Phi59}}. However, none
of these approaches can be extended to nonlinear Langevin systems. In this section
we focus on methods that work for nonlinear systems.
\vspace{2mm}

Once the parameters have been estimated, one numerically solves the
estimated system to make predictions. In this paper, to make predictions
for time $t>Nh$ (where $N$ is the number of observations), we use the
initial condition $\left(x_{Nh},\hat{y}_{N}\right)$ in solving the estimated
system, with $\hat{y}_{N}$ being an estimate of $y_{Nh}$ based on observations $x$. Since the
system is stochastic, we use an ``ensemble forecasting'' method to make
predictions.  We start a number of trajectories from the same
initial condition, and evolve each member of this ensemble
independently.  The ensemble characterizes the possible motions of the
particle conditional on past observations, and the ensemble mean
provides a specific prediction.  For the purpose of short-term
prediction, the estimated system can be solved with small time steps,
hence a low order scheme such as the Euler scheme may be used.

However, in many practical applications, the true system is unknown \cite{CL15, LLC15}, and one has to validate the continuous-time model by its ability to reproduce the long-term statistics of data.  For this purpose, one has to compute the ergodic limits of the estimated system. The Euler scheme may be numerically
unstable when the system is not globally Lipschitz, and a better scheme
such as implicit Euler (see e.g.\cite{MSH02, Tal02, MST10}) or the
quasi-symplectic integrator \cite{MT07}, is needed. In our study, the
Euler scheme is numerically unstable, while the \Ito-Taylor scheme of
strong order 2.0 in (\ref{IT2xy}) produces long-term statistics close to
those produced by the implicit Euler scheme. We use the \Ito-Taylor
scheme, since it has the advantage of being explicit and
was used in \cite{PSW09}.

In summary, the continuous-time approach uses the following algorithm to
generate a forecasting ensemble of trajectories.
\begin{algorithm}[Continuous-time approach] \rm
With data $\left\{ x_{nh}\right\} _{n=1}^{N}$, \\
\begin{tabular}{rp{5.5in}}
Step 1. & Estimate the parameters using $(\ref{MLE})$;\\
Step 2. & Select a numerical scheme for the SDE, e.g. the \Ito-Taylor scheme in the appendix;  \\
Step 3. & Solve the SDE $\left( \ref{Keq}\right)$ with estimated parameters, using small time steps $dt$ and initial data $\left( x_{Nh},\frac{x_{Nh}-x_{Nh-h}}{h}\right) $, to generate the forecasting ensemble. \\
\end{tabular}
\end{algorithm}

%%%%%%%%%%%%%%%%%%%%%%%%%%%%%%%%
\section{The discrete-time approach \label{sec3}}

\subsection{NARMA representation \label{sec_Narma}}

In the discrete-time approach, the goal is to infer a discrete-time predictive 
model for $x$ from the data.
Following \cite{CL15}, we choose a discrete-time system in the form
of a nonlinear autoregression moving average (NARMA) model of the following form:
\begin{eqnarray}
X_{n}&=&\Phi_{n}+\xi _{n}, \label{Narma} \\
 \Phi_{n}  &:=& \mu
+\sum_{j=1}^{p}a_{j}X_{n-j}+\sum_{k=1}^{r}b_{k}Q_{k}(X_{n-p:n-1},\xi
_{n-q:n-1})+\sum_{j=1}^{q}c_{j}\xi _{n-j} , \label{Phi} 
\end{eqnarray}%
where $p$ is the order of the autoregression, $q$ is the order of the
moving average, and the $Q_{k}$ are given nonlinear functions
(see below) of $\left( X_{n-p:n-1},\xi _{n-q:n-1}\right) $. Here
$\left\{ \xi _{n}\right\} $ is a sequence of i.i.d Gaussian random
variables with mean zero and variance $c_{0}^{2}$ (denoted by
$\mathcal{N}(0,c_{0}^{2})$). The numbers $p$, $q$, $r$, as well
  as the coefficients $a_j$, $b_j$, and $c_j$ are to be determined from
  data.

A main challenge in designing NARMA models is the choice of the
functions $Q_k$, a process we call ``structure selection'' or
``structure derivation''.  Good structure design leads to models that
fit data well and have good predictive capabilities.  Using too many
unnecessary terms, on the other hand, can lead to overfitting
 or inefficiency, while too few terms can lead to
  an ineffective model.
As before, we assume that a parametric family containing the true model
is known, and we show that suitable structures for NARMA can be derived
from numerical schemes for solving SDEs. We propose the following
practical criteria for structure selection: (i) the model should be
numerically stable; (ii) we select the model that makes the best
predictions (in practice, the predictions can be tested using the given
data.); (iii) the large-time statistics of the model should agree with
those of the data. These criteria are not sufficient to uniquely specify
a viable model, and we shall return to this issue when we discuss the numerical experiments.

Once the $Q_k$ have been chosen, the coefficients $\left( a_{j},b_{j},c_{j}\right)$ are estimated from
data using the following conditional likelihood method. Conditional on
$\xi _{1},\dots ,\xi _m$, the log-likelihood of $\left\{
X_{n}=x_{nh}\right\} _{n=m+1}^{N}$ is%
\begin{equation}
\label{log-likelihood}
L_{N}(\vartheta |\xi _{1},\dots ,\xi _{m})=\sum_{n=m+1}^{N}\frac{\left(
X_{n}-\Phi _{n}\right) ^{2}}{2c_{0}^{2}}+\frac{N-q}{2} \log c_{0}^{2},
\end{equation}%
where $m=\max\{p,q\}$ and $\vartheta =\left(
a_{j},b_{j},c_{j},c_{0}^{2}\right)$, and $\Phi_n$ is defined in
  Eq.~(\ref{Phi}).
%
% and 
% \[
%\Phi_{n}=\sum_{j=1}^{p}a_{j}X_{n-j}+\sum_{k=1}^{r}b_{k}Q_{k}(X_{n-p:n-1},\xi_{n-q:n-1})+\sum_{j=1}^{q}c_{j}\xi _{n-j}.
%\]
%
The log-likelihood is computed as follows. Conditionally on given values
of $\{\xi _{1},\dots ,\xi _m\}$, one can compute $\Phi_{m+1}$ from data $
\left\{ X_{n}=x_{nh}\right\} _{n=1}^m$ using  Eq.~(\ref{Phi}). 
With the value of $\xi_{m+1}$ following from (\ref{Narma}), 
one can then compute $\Phi_{m+2}$. Repeating this recursive procedure, one obtains the values of $\{\Phi_{n}\}_{n=m+1}^{N}$ that are needed to evaluate the log-likelihood.
The estimator of the parameter $\vartheta
=\left( a_{j},b_{j},c_{j},c_{0}^{2}\right) $ is the minimizer of the
log-likelihood
\begin{equation*}
\hat{\vartheta}_{N}=\arg \min_{\vartheta }L_{N}(\vartheta |\xi _{1},\dots
,\xi _{m}).
\end{equation*}%
If the system is ergodic, the conditional maximum likelihood estimator
$\hat{\vartheta}_{N} $ can be proved to be consistent (see
e.g.\cite{And70,Ham94}), which means that it converges almost surely to the true parameter value as $N\to \infty$. Note that the estimator requires the
  values of $\xi_1,\cdots,\xi_m$, which is in general not available.
  But ergodicity implies that if $N$ is large, $\hat{\vartheta}_{N}$
  forgets about the values of $\xi_1,\cdots,\xi_m$ quickly anyway, and
  in practice, we can simply set $\xi_1 =\dots =\xi _m=0$. Also, in practice, we initialize the
optimization with $c_1=\dots= c_q=0$ and with the values of $(a_j,b_j)$
computed by least-squares.

Note that in the case $q=0$, the estimator is the same as the nonlinear
least-squares estimator. The noise sequence $\left\{ \xi _{n}\right\} $
does not have to be Gaussian for the conditional likelihood method to
work, so long as the expression in Eq.~(\ref{log-likelihood}) is
  adjusted accordingly.

In summary, the discrete-time approach uses the following algorithm to a
generate a forecasting ensemble.

\begin{algorithm}[Discrete-time approach] \rm
With data $\left\{ x_{nh}\right\} _{n=1}^{N}$,\\
\begin{tabular}{rp{5.5in}}
Step 1. & Find possible structures for NARMA;\\
Step 2. & Estimate the parameters in NARMA for each possible structure;\\
Step 3. & Select the structure that fits the data  best, in the sense that it reproduces best the long-term
  statistics and makes the best predictions;\\
Step 4. & Use the resulting model to generate a forecasting ensemble.\\
\end{tabular}
\end{algorithm}

\subsection{Structure derivation for the linear Langevin equation} \label{CAR}

The main difficulty in the discrete-time approach is the derivation of the
structure of the NARMA model. In this section we discuss how to derive this structure
from the SDEs, first in the linear case. 

For the linear Langevin equation, the
discrete-time system should be linear. Hence we set $r=0$ in (\ref{Narma})
and obtain an ARMA($p,q$) model.  
The linear Langevin equation
\begin{equation}
\left\{ \begin{aligned} dx&=ydt, \\ dy&=(-\gamma y - \alpha x)dt+\sigma
dB_t,\end{aligned}\right.  \label{OU}
\end{equation}%
can be solved analytically. The solution $x_{t}$ at discrete times satisfies
(see Appendix~\ref{solutionLLE}) 
\begin{equation}
x_{\left( n+2\right) h}=a_{1}x_{\left( n+1\right)
h}+a_{2}x_{nh}-a_{22}W_{n+1,1}+W_{n+2,1}+a_{12}W_{n+1,2},  \label{OUts}
\end{equation}
where $\left\{ W_{n,i}\right\} $ are defined in (\ref{Vi}), and
\begin{equation} \label{a1a2}
a_{1}=\mathrm{trace}(e^{\mathbf{A}h}), a_{2}= - e^{-\gamma h}, 
a_{ij}=\left( e^{\mathbf{A}h}\right) _{ij}, \text{for } \mathbf{A}=\left( 
\begin{array}{cc}
0 & 1 \\ 
-\alpha & -\gamma
\end{array}
\right).
\end{equation}

The process $\left\{ x_{nh}\right\} $ defined in Eq.~(\ref{OUts}) is,
strictly speaking, not an ARMA process (see Appendix \ref{section:arma} for all
relevant, standard definitions used in this section), because $\left\{ W_{n,1}\right\} _{n=1}^{\infty }$ and
$\left\{ W_{n,2}\right\} _{n=1}^{\infty }$ are not linearly dependent
and would require at least two independent noise sequences to represent,
while an ARMA process requires only one.  However, as the following
proposition shows, there is an ARMA process with the same distribution
as the process $\left\{ x_{nh}\right\} $.  Since the minimum
mean-square-error state predictor of a stationary Gaussian process
depends only on its autocovariance function (see, e.g., \cite[Chapter
  5]{BD91}), an ARMA process equal in distribution to the discrete-time
Langevin equation is what we need here.

\begin{proposition}
\label{prop} The ARMA$\left( 2,1\right) $ process
\begin{equation}
X_{n+2}=a_1 X_{n+1}+ a_2 X_{n}+W_{n}+\theta _{1}W_{n-1},  \label{OU_ARMA21}
\end{equation}
where $a_1,a_2$ are given in $(\ref{a1a2})$ and the $\left\{ W_{n}\right\} $ are i.i.d $\mathcal{N}(0,\sigma _{W}^{2})$, is the unique process in the family of invertible ARMA processes that has the same distribution as the process $\left\{ x_{nh}\right\} $.  Here $\sigma _{W}^{2}$ and $\theta _{1}$ ($\theta_1 <1$ so that the process is invertible) satisfy the equations 
\begin{eqnarray*}
\sigma _{W}^{2}\left( 1+\theta _{1}^{2}+\theta _{1}a_{1}\right)
&=&\gamma _{0}-\gamma _{1}a _{1}-\gamma _{2}a _{2}, \\
\sigma _{W}^{2}\theta _{1} &=&\gamma _{1}\left( 1-a _{2}\right) -\gamma
_{0}a _{1},
\end{eqnarray*}
where $\left\{\gamma _{j}\right\} _{j=0}^{2 }$ are the autocovariances of the process $\left\{ x_{nh}\right\} $ and are given in Lemma $\ref{eigen}$.
\end{proposition}

\begin{proof}
 Since the stationary process $\left\{ x_{nh}\right\} $ is a centered
Gaussian process, we only need to find an ARMA$\left( p,q\right) $\ process
with the same autocovariance function as $\left\{ x_{nh}\right\} $. The
autocovariance function of $\left\{ x_{nh}\right\} $, denoted by $\left\{
\gamma _{n}\right\} _{n=0}^{\infty }$, is given by (see Lemma \ref{eigen})
\begin{equation*} 
\gamma _{n}=\gamma _{0}\times \left\{ 
\begin{array}{cc}
\frac{1}{\lambda _{1}-\lambda _{2}}(\lambda _{1}e^{\lambda _{2}nh}-\lambda
_{2}e^{\lambda _{1}nh}),~ & \text{if }\gamma ^{2}-4\alpha \neq 0; \\[1ex] 
e^{\lambda _{0}nh}(1-\lambda _{0}nh), & \text{if }\gamma ^{2}-4\alpha =0,%
\end{array}%
\right.
\end{equation*}%
where $(\lambda_1, \lambda_2, \text{ or } \lambda_0 )$ are the roots of the characteristic polynomial $\lambda^2+\gamma \lambda +\alpha=0$ of the matrix $\mathbf{A}$ in (\ref{a1a2}). 

On the other hand, the autocovariance function of an ARMA$\left( p,q\right)$
process
\begin{equation*}
X_{n}-\phi _{1}X_{n-1}-\cdots -\phi _{p}X_{n-p}=W_{n}+\theta
_{1}W_{n-1}+\cdots +\theta _{q}W_{n-q}, 
\end{equation*}%
denoted as $\left\{ \gamma \left( n\right) \right\}
_{n=0}^{\infty }$, is given by (see Eq.~(\ref{gamma_roots}))
\begin{equation*}
\gamma (n)=\sum_{i=1}^{k}\sum_{j=0}^{r_{i}-1}\beta _{ij}n^{j}\zeta _{i}^{-n}%
~\text{,~~~for }n\geq \max\{p,q+1\}-p, 
\end{equation*}%
where $\left( \zeta _{i},i=1,\dots ,k\right) $ are the distinct zeros of $\phi
(z):=1-\phi _{1}z-\cdots -\phi _{p}z^{p}$, and $r_{i}$ is the
multiplicity of $\zeta _{i}$ (hence $\sum_{i=1}^{k}r_{i}=p$), and $\{\beta _{ij}\}$ are constants.

Since $\left\{ \gamma_{n}\right\} _{n=0}^{\infty }$ only provides two possible roots, $\zeta
_{i}=e^{-\lambda _{i}h}$ or $\zeta _{i}=e^{-\lambda _{0}h}$ for $i=1,2$, the
order $p$ must be that $p=2$. From these two roots, one can compute the coefficients $\phi _{1}$ and $\phi _{2}$ in the ARMA$\left( 2,q\right) $ process:%
\begin{equation*}
\phi _{1}=\zeta _{1}^{-1}+\zeta _{2}^{-1}=\mathrm{trace}(e^{\mathbf{A}h})=a_1,~\phi
_{2}=-\zeta _{1}^{-1}\zeta _{2}^{-1}=-e^{-\gamma h}=a_2.
\end{equation*}%
Since $\gamma _{k}-\phi _{1}\gamma _{k-1}-\phi _{2}\gamma _{k-2}= 0$ for any $k\geq 2$, we have $q\leq1$. Since $\gamma _{1}-\phi _{1}\gamma _{0}-\phi _{2}\gamma _{1}\neq 0$, Example %
\ref{exmpl20} indicates that $q\neq 0$. Hence $q=1$ and the above ARMA($2,1$) is the unique process in the family of invertible ARMA($p,q$) processes that has the same distribution as $\left\{ x_{nh}\right\} $. The equations for $\sigma _{W}^{2}$ and $\theta _{1}$ follow from Example \ref{exmpl21}.  
\end{proof}

\bigskip
This proposition indicates that the discrete-time system for the linear
Langevin system should be an ARMA$\left( 2,1\right)$ model.

\begin{example} \label{th-value}
{\rm
Suppose $\Delta:=\gamma^2-4\alpha<0$. Then the parameters in the ARMA$(2,1)$ process $(\ref{OU_ARMA21})$ are given by $a_1=2e^{-\frac{\gamma}{2} h}\cos(\frac{\sqrt{-\Delta}}{2} h)$, 
$ a_2=-e^{-\gamma h}$ and
\begin{eqnarray*}
\theta_1 = \frac{c-a_1-\sqrt{(c-a_1)^2-4}}{2}, \ \ 
\sigma_w^2 = \frac{\gamma_1(1-a_2)-\gamma_0 a_1}{\theta_1}.
\end{eqnarray*}
where $c=\frac{\gamma_0-\gamma_1 a_1-\gamma_2 a_2}{\gamma_1(1-a_2)-\gamma_0 a_1}$, and $
\gamma_n= \frac{\sigma^2}{2\gamma \alpha} \left(\cos(\frac{\sqrt{-\Delta}}{2} nh) + \frac{\gamma}{\sqrt{-\Delta}} \sin(\frac{\sqrt{-\Delta}}{2} nh)\right)$
for $n\geq 0$. 
}
\end{example}

\begin{remark}
{\rm
  The maximum likelihood estimators of ARMA parameters can also be computed using a state-space representation and a Kalman recursion {\rm(}see e.g. {\rm\cite{BD91})}.  This approach is essentially the same as the conditional likelihood method in our discrete-time approach.
}  
\end{remark}

\begin{remark}{\rm The proposition indicates that the parameters in the linear Langevin equation can also be computed from the ARMA$(2,1)$ estimators, because from the proof we have $\gamma = -\frac{\ln (-a_2)}{h} =-\lambda_1-\lambda_2$, $\alpha = \lambda_1 \lambda_2$, and $\sigma^2=2\gamma\alpha \sigma_W^2$, where   $(\lambda_i, i=1, 2)$ satisfies that $(e^{-\lambda_i h}, i=1, 2)$ are the two roots of $\phi(z) = 1- a_1 z - a_2 z$. 
}\end{remark}

\subsection{Structure derivation for the Kramers oscillator}

For nonlinear Langevin systems, in general there is no analytical
solution, so the approach of Section~\ref{CAR}
  cannot be used.  Instead, we derive structures from the numerical
schemes for solving stochastic differential equations. For simplicity,
we choose to focus on explicit terms in a discrete-time system, so
implicit schemes (in e.g. \cite{MSH02, Tal02, MT07}) are not
suitable. Here we focus on deriving structures from two explicit
schemes: the Euler--Maruyama scheme and the \Ito-Taylor scheme of order
2.0; see Appendix~\ref{schemeLE} for a brief review of these schemes. As
mentioned before, we expect our approach to extend to other explicit
schemes, e.g., that of~\cite{AM11}. While we consider specifically
Eq.~(\ref{Keq}), the method used in this section extends to situations
when $f(x,y)$ is such that Eq.~(\ref{eq2.1}) can be rewritten in form
Eq.~(\ref{Keq}) and its higher-dimensional analogs by a change of
variables.

As warm-up, we begin with the Euler--Maruyama scheme. Applying
this scheme (\ref{EMxy}) to the system (\ref{Keq}), we find:
\begin{eqnarray*}
x_{n+1} &=&x_{n}+y_{n}h, \\
y_{n+1} &=&y_{n}(1-\gamma h)-hV^{\prime }(x_{n})+W_{n+1},
\end{eqnarray*}%
where $W_{n}=\sigma h^{1/2}\zeta _{n}$,$\,$\ with $\left\{ \zeta
_{n}\right\} $ is an i.id. sequence of $\mathcal{N}(0,1)$ random
variables. Straightforward substitutions yield a closed system for $x$\begin{equation*}
  x_{n}=(2-\gamma h)x_{n-1}-(1-\gamma h)x_{n-2}-h^{2}V^{\prime
  }(x_{n-2})+hW_{n-1}.
\end{equation*}%
Since $V^{\prime }(x)=\beta x^{3}-\alpha x$, this leads to the following
possible structure for NARMA: \\% \newline
\textbf{Model (M1):} \vspace{-3mm}
\begin{equation}
X_{n}=a_{1}X_{n-1}+a_{2}X_{n-2}+b_{1}X_{n-2}^{3}+\xi
_{n}+\sum_{j=1}^{q}c_{j}\xi _{n-j}+\mu .  \label{Narma1}
\end{equation}

Next, we derive a structure from the \Ito-Taylor scheme of order 2.0. Applying the scheme (\ref{IT2xy}) to the system (\ref{Keq}), we find
\begin{eqnarray*}
x_{n+1} &=&x_{n}+h\left( 1-0.5\gamma h\right) y_{n}- 0.5h^{2}V^{\prime }\left(
x_{n}\right) +Z_{n+1}, \\
y_{n+1} &=&y_{n}\left[ 1-\gamma h+ 0.5\gamma ^{2}h^{2}- 0.5h^{2}V^{\prime \prime
}\left( x_{n}\right) \right] -h\left( 1- 0.5\gamma h\right) V^{\prime }\left(
x_{n}\right) +W_{n+1}-\gamma Z_{n+1},
\end{eqnarray*}%
where $Z_{n}=\sigma h^{3/2}\left( \zeta _{n}+ \eta_n/\sqrt{3}\right) $, with $\left\{ \eta _{n}\right\} $ being an
i.id. $\mathcal{N}(0,1)$ sequence independent of $\left\{ \zeta
_{n}\right\} $. Straightforward substitutions yield a
closed system for $x:$
\begin{eqnarray*}
 x_{n} &=&x_{n-1}\left[ 2-\gamma h+ 0.5\gamma ^{2}h^{2}-h^{2}V^{\prime
     \prime }\left( x_{n-2}\right) \right] - 0.5h^{2}V^{\prime }\left(
 x_{n-1}\right) +Z_{n} \\
 &&+\left[ 1-\gamma h+ 0.5\gamma ^{2}h^{2}- 0.5h^{2}V^{\prime \prime }\left(
   x_{n-2}\right) \right] \left( -x_{n-2}+ 0.5h^{2}V^{\prime }\left(
 x_{n-2}\right) -Z_{n-1}\right) \\
 &&-h^{2}\left( 1- 0.5\gamma h\right) ^{2}V^{\prime }\left( x_{n-2}\right)
  +h\left( 1- 0.5\gamma h\right) \left( W_{n-1}-\gamma Z_{n-1}\right) .
\end{eqnarray*}%
Note that $W_{n}$ is of order $h^{1/2}$ and $Z_{n}$ is of order $h^{3/2}$.
Writing the terms in descending order, we obtain 
\begin{eqnarray}
x_{n} &=&\left( 2-\gamma h+ 0.5\gamma ^{2}h^{2}\right) x_{n-1}-\left( 1-\gamma
h+ 0.5\gamma ^{2}h^{2}\right) x_{n-2} \label{IT2_narma} \\
&&+Z_{n}-Z_{n-1}+h\left( 1- 0.5\gamma h\right) W_{n-1}- 0.5h^{2}V^{\prime }\left(
x_{n-1}\right) + 0.5h^{2}V^{\prime \prime }\left( x_{n-2}\right) \left(
x_{n-1}-x_{n-2}\right) \notag \\
&&+ 0.5\gamma h^{3}V^{\prime }\left( x_{n-2}\right) + 0.5h^{2}V^{\prime \prime
}\left( x_{n-2}\right) Z_{n-1}- 0.5h^{4}V^{\prime \prime }\left( x_{n-2}\right)
V^{\prime }\left( x_{n-2}\right) . \notag 
\end{eqnarray}%
This equation suggests that $p=2$ and $q=0$ or $1$. The noise term
$Z_{n}-Z_{n-1}+h\left( 1- 0.5\gamma h\right) W_{n-1}$ is of order $h^{1.5}$,
and involves two independent noise sequences $\left\{ \zeta _{n}\right\}
$ and $\left\{ \eta _{n}\right\} $, hence the above equation for $x_{n}$
is not a NARMA model. However, it suggests possible structures for NARMA
models. In comparison to model (M1), the above equation has (i)
different nonlinear terms of order $h^{2}$: $h^{2}V^{\prime }\left(
x_{n-1}\right) $\ and $h^{2}V^{\prime \prime }\left( x_{n-2}\right)
\left( x_{n-1}-x_{n-2}\right) $; (ii) additional nonlinear terms of
orders three and larger: $h^{3}V^{\prime }\left( x_{n-2}\right) $,
$h^{2}Z_{n-1}V^{\prime \prime }\left( x_{n-2}\right) $, and
$h^{4}V^{\prime \prime }\left( x_{n-2}\right) V^{\prime }\left(
x_{n-2}\right) $.  It is not clear which terms should be used, and one
may be tempted to include as many terms as possible.
However, this can lead to overfitting. Hence, we consider
different structures by successively adding more and more
terms, and select the one that fits data the best. Using the fact
that $V^{\prime }(x)=\beta x^{3}-\alpha x$, these terms lead to the
following possible structures for NARMA: \\ \textbf{Model (M2):}%
\begin{equation*}
X_{n}=a_{1}X_{n-1}+a_{2}X_{n-2}+b_{1}X_{n-1}^{3}+\underbrace{b_{2}X_{n-2}^{2}\left(
X_{n-1}-X_{n-2}\right)} +\xi _{n}+\sum_{j=1}^{q}c_{j}\xi _{n-j}+\mu ,
\end{equation*}%
where $b_{1}$ and $b_{2}$ are of order $h^{2}$, and $q\geq 0$; \newline
\textbf{Model (M3):}%
\begin{eqnarray*}
X_{n}&=&a_{1}X_{n-1}+a_{2}X_{n-2}+b_{1}X_{n-1}^{3}+\underbrace{b_{2}X_{n-2}^{2}\left(
X_{n-1}-X_{n-2}\right)} \\
&& +\underbrace{b_{3}X_{n-2}^{3}}+\xi _{n}+\sum_{j=1}^{q}c_{j}\xi
_{n-j}+\mu ,
\end{eqnarray*}%
where $b_{3}$ is of order $h^{3}$, and $q\geq 0$; \newline
\textbf{Model (M4):}%
\begin{eqnarray*}
X_{n}
&=&a_{1}X_{n-1}+a_{2}X_{n-2}+b_{1}X_{n-1}^{3}+\underbrace{b_{2}X_{n-2}^{2}X_{n-1}}+\underbrace{b_{3}X_{n-2}^{3}}+\underbrace{b_{4}X_{n-2}^{5}}
\\
&&+\underbrace{b_{5}X_{n-2}^{2}\xi _{n-1}}+\xi _{n}+\sum_{j=1}^{q}c_{j}\xi _{n-j}+\mu ,
\end{eqnarray*}%
where $b_{4}$ is of order $h^{4}$, and $b_{5}$ is of order $h^{3.5}$,
and $q\geq 1$. (For the reader's convenience, we have highlighted
  all higher-order terms derived from $V'(x)$.)

From the model (M2)--(M4), the number of nonlinear terms increases as their
order increases in the numerical scheme. Following \cite{CL15,LLC15}, we use only the form of the terms derived from numerical analysis, and not their coefficients; we estimate new coefficients from data.

\section{Numerical study} \label{sec4}
 
We test the continuous and discrete-time
approaches for data sets with different observation intervals
$h$. The data are generated by solving the general Langevin
Eq.~(\ref{Keq}) using a second-order \Ito-Taylor scheme, with a small step size $dt=1/1024$, and making
observations with time intervals $h=1/32,1/16$, and $1/8$; the value of
time step $dt$ in the integration has been chosen to be sufficiently small to guarantee reasonable
accuracy.  For each one of the data sets, we estimate the parameters in
the SDE and in the NARMA models. We then compare the estimated SDE and
the NARMA model by their ability to reproduce long-term statistics and
to perform short-term prediction.

\subsection{The linear Langevin equation}

We first discuss numerical results in the linear case. 
Both approaches start by computing the estimators. The estimator $\hat{\theta}=\left( \hat{\gamma},\hat{\alpha},\hat{\sigma}\right) $ of the
parameters $\left( \gamma ,\alpha ,\sigma \right) $ of the linear Langevin
Eq.~(\ref{OU}) is given by

\begin{equation}
\hat{\theta}=\arg \min_{\theta =\left( \gamma ,\alpha ,\sigma \right) }\left[
\sum_{n=1}^{N-3}\frac{3}{2}\frac{\left[ \hat{y}_{n+2}-\hat{y}_{n+1}+h(\gamma 
\hat{y}_{n}+\alpha x_{n})\right] ^{2}}{h\sigma ^{2}}+\left( N-3\right) \log
\sigma ^{2}\right] ,  \label{MLE_LE}
\end{equation}%
where $\hat{y}_{n}$ is computed from data using (\ref{xBar}).

Following Eq.~(\ref{OU_ARMA21}), we use the ARMA$\left( 2,1\right) $
model in the discrete-time approach: 
\begin{equation*}
X_{n+2}=a_{1}X_{n+1}+a_{2}X_{n}+W_{n}+\theta _{1}W_{n-1},
\end{equation*}%
We estimate the parameters $a_{1},a_{2}, \theta _{1}$, and $\sigma
_{W}^{2}$ from data using the conditional likelihood method of
Section~\ref{sec_Narma}.

First, we investigate the reliability of the estimators. A hundred
simulated data sets are generated from Eq.~(\ref{OU}) with true
parameters $\gamma =0.5$, $\alpha =4$, and $\sigma =1$, and with initial
condition $x_{0}=y_{0}=\frac{1}{2}$ and time interval $[0,10^{4}]$. The
estimators, of $\left( \gamma ,\alpha ,\sigma \right) $ in the linear
Langevin equation and of $\left( a_{1},a_{2},\theta _{1},\sigma
_{W}\right) $ in the ARMA$\left( 2,1\right) $ model, are computed for
each data set. Empirical mean and standard deviation of the estimators
are reported in Table \ref{tab:ST_OU} for the continuous-time approach,
and Table \ref{tab:narmaOU} for the discrete-time approach. In the
continuous-time approach, the biases of the estimators grow as $h$
increases. In particular, large biases occur for the estimators of
$\gamma$: the bias of $\hat{\gamma}$ increases from $0.2313$ when
$h=1/32$ to $0.4879$ when $h=1/8$, while the true value is $\gamma
=0.5$; similarly large biases were also noticed in
\cite{ST12}. In contrast, the biases are much smaller for the
discrete-time approach. The \textquotedblleft theoretical
value\textquotedblright\ (denoted by ``T-value'') of
$a_{1},a_{2}\,,\theta _{1}$ and $\sigma _{W}^{2}$ are computed
analytically as in Example \ref{th-value}. Table \ref{tab:narmaOU} shows
that the estimators in the discrete-time approach have negligible
differences from the theoretical values.

In practice, the above test of the reliability of estimators cannot be
performed, because one has only a single data set and the true system
that generated the data is unknown.

%%%%%%%%%%   % alpha =4; sigma =1
\begin{table}[tbp]
\caption{Mean and standard deviation of the estimators of the parameters $(%
\protect\gamma ,\protect\alpha ,\protect\sigma )$ of the linear Langevin
equation in the continuous-time approach, computed on 100 simulations. }
\label{tab:ST_OU}\centering
\begin{tabular}{clccc}
\hline
Estimator & True value & $h=1/32$ & $h=1/16$ & $h=1/8$ \\ \hline
$\hat{\gamma}$ & $0.5$ & 0.7313 (0.0106) & 0.9538 (0.0104) & 1.3493 (0.0098)
\\ 
$\hat{\alpha}$ & $4$ & 3.8917 (0.0193) & 3.7540 (0.0187) & 3.3984 (0.0172)
\\ 
$\hat{\sigma}$ & $1$ & 0.9879 (0.0014) & 0.9729 (0.0019) & 0.9411 (0.0023)
\\ \hline
\end{tabular}%
\end{table}

\begin{table}[tbp]
\caption{Mean and standard deviation of the estimators of the parameters $(a_{1},a_{2},\protect\theta _{1},\protect\sigma _{W})$ of the ARMA$(2,1)$
model in the discrete-time approach, computed on 100 simulations. The
theoretical value (denoted by T-value) of the parameters are computed from
proposition \protect\ref{prop}. }
\label{tab:narmaOU}\centering
\begin{tabular}{ccccccc}
\hline
Estimator & \multicolumn{2}{c}{$h=1/32$} & \multicolumn{2}{c}{$h=1/16$}
& \multicolumn{2}{c}{$h=1/8$} \\ 
& T-value & Est. value & T-value & Est. value & T-value & Est. value \\ 
\hline
$\hat{a}_{1}$         & 1.9806 & 1.9807 (0.0003) & 1.9539  & 1.9541 (0.0007) & 1.8791 & 1.8796 (0.0014)
\\ 
$-\hat{a}_{2}$       & 0.9845 & 0.9846 (0.0003)  &0.9692  & 0.9695 (0.0007) &  0.9394  & 0.9399 (0.0014)
\\ 
$\hat{\theta}_{1}$   &0.2681 & 0.2667 (0.0017) & 0.2684  & 0.2680 (0.0025) & 0.2698   & 0.2700 (0.0037) \\ 
$\hat{\sigma}_{W}$&0.0043 & 0.0043 (0.0000) & 0.0121  & 0.0121 (0.0000) & 0.0336   & 0.0336 (0.0001) \\ \hline
\end{tabular}%
\end{table}
\setlength{\tabcolsep}{6pt}

We now compare the two approaches in a practical setting, by assuming
that we are only given a single data set from discrete observations of a
long trajectory on time interval $[0,T]$ with $T=2^{17}\approx 1.31\times
10^5$. We estimate the parameters in the SDE and the ARMA model, and
again investigate the performance of the estimated SDE and ARMA model
in reproducing long-term statistics and in predicting the short-term
evolution of $x$. The long-term statistics are computed by
time-averaging. The first half of the data set is used to compute the
estimators, and the second half of the data set is used to test the
prediction.

The long-term statistics, i.e., the empirical probability density
function (PDF) and the autocorrelation function (ACF), are shown in
Figure \ref{fig_OU_acf}. For all the three values of $h$, the ARMA models reproduce the empirical PDF and ACF almost perfectly. The estimated SDEs miss the spread of the PDF and the amplitude of oscillation in the ACF, and these error become larger as $h$ increases.

%%%%%%%%%%%%%%%%%%
%\begin{comment}
\begin{figure}[tbp]
\begin{center}
\begin{tabular}{c}
\resizebox{0.92\textwidth}{!}{\includegraphics{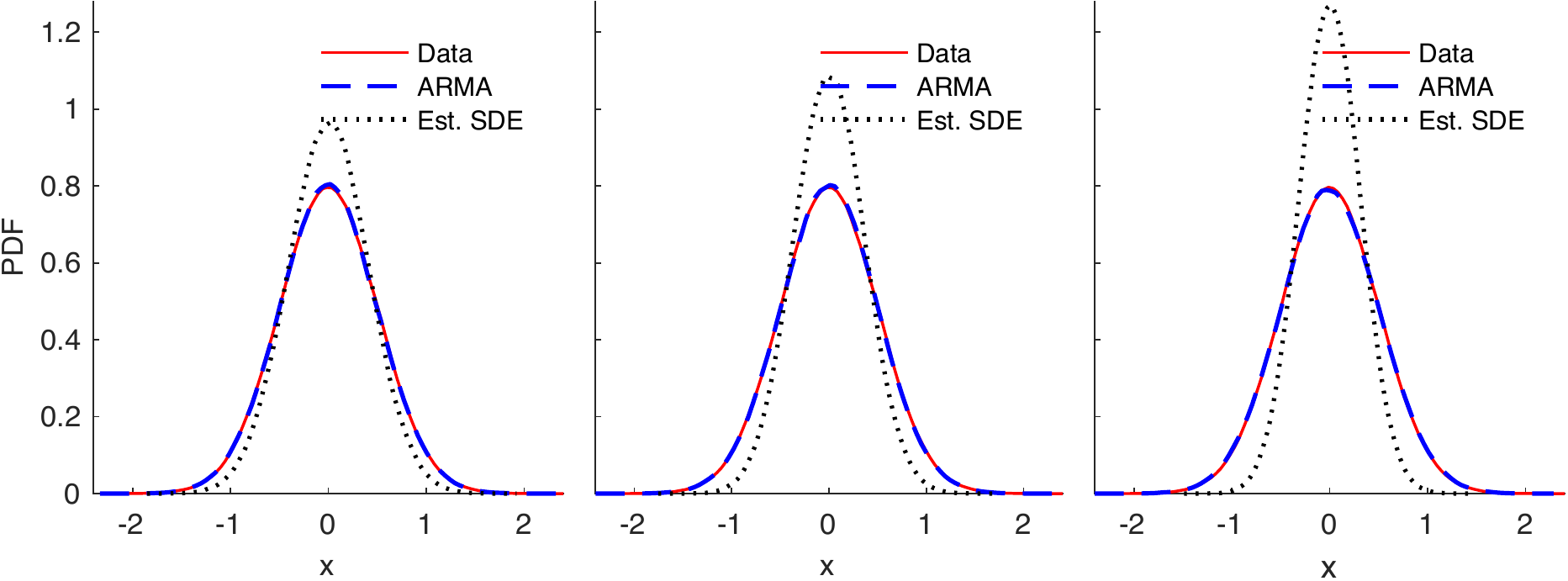}} \\
\resizebox{0.92\textwidth}{!}{\includegraphics{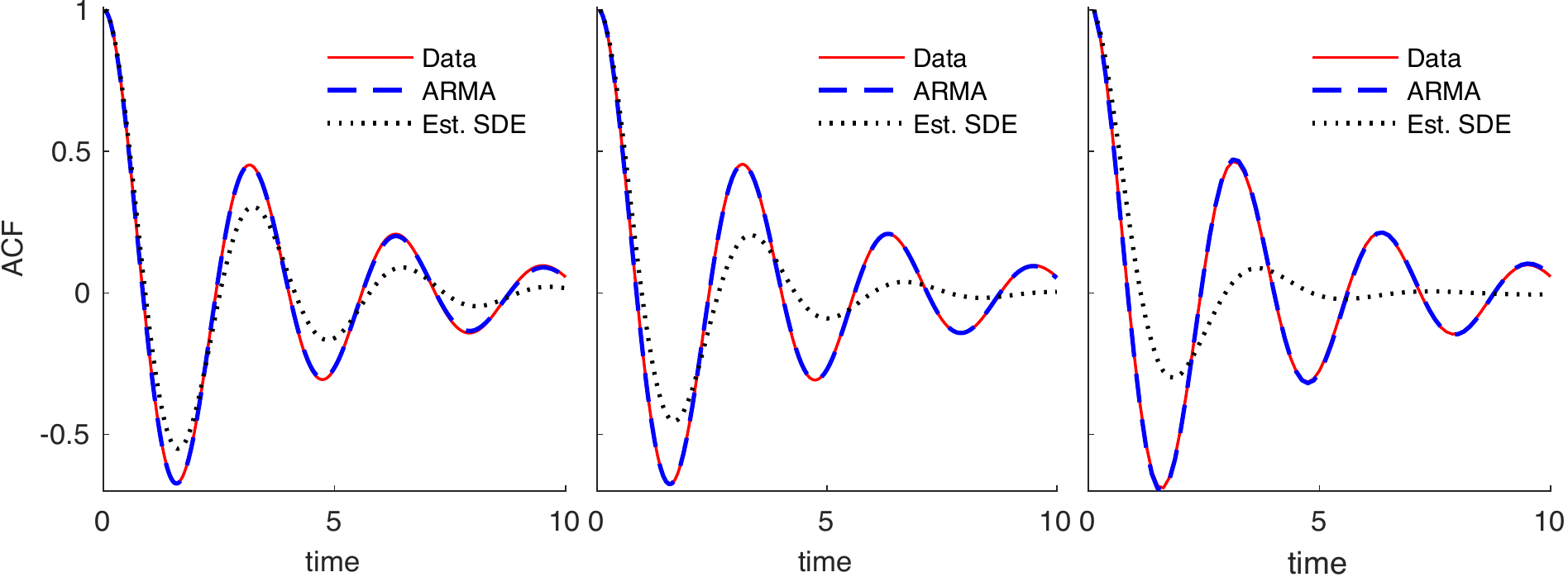}} 
\end{tabular}
\end{center}
\vspace{-6mm}
\caption{Empirical PDF and ACF  of the ARMA$(2,1)$ models and the estimated linear Langevin system (denoted by Est.~SDE), in the cases $h=1/32$ (left column), $h=1/16$  (middle column) and $h=1/8$  (right column). The ARMA models reproduce the PDF and ACF almost perfectly, much better than the estimated SDEs. }
\label{fig_OU_acf}
\end{figure}
%\end{comment}
%%%%%%%%%%%%%%%%%%%%%

Next, we use an ensemble of trajectories to predict the motion of
$x$. For each ensemble, we calculate the mean trajectory and compare it
with the true trajectory from the data. We measure the performance of
the prediction by computing the root-mean-square-error (RMSE) of a large number of ensembles as follows:
take $N_{0}$ short pieces of data from the second half of the long
trajectory, denoted by $\left\{ \left( x_{\left( n_{i}+1\right) h},\dots
,x_{\left( n_{i}+K\right) h}\right) \right\} _{i=1}^{N_{0}}$, where
$n_{i}=Ki$. For each short piece of data $\left( x_{\left(
  n_{i}+1\right) h},\dots ,x_{\left( n_{i}+K\right) h}\right) $, we
generate $N_{ens}$ trajectories $\left\{ \left( X_{1}^{i,j},\dots
,X_{K}^{i,j}\right) \right\} _{j=1}^{N_{ens}}$ using a prediction system
(i.e., the NARMA($ p, q$), the estimated Langevin system, or the true
Langevin system), starting all ensemble members from the same
several-step initial condition $\left( x_{\left( n_{i}+1\right) h},\dots
,x_{\left( n_{i}+m\right) h}\right) $, where $m=2\max \left\{
p,q\right\} +1$. For the NARMA$\left( p,q\right) $ we start with $\xi
_{1}=\cdots =\xi _{q}=0$. For the estimated Langevin system and the true
Langevin system, we start with initial condition $\left(
x_{\left( n_{i}+m\right) h}, \hat{y}_{n_i}\right)$ with $\hat{y}_{n_i}=\frac{x_{\left( n_{i}+m\right) h}-x_{\left(
    n_{i}+m-1\right) h}}{h}$ and solve the
  equations using the \Ito-Taylor scheme of order 2.0 with a time step $dt=1/64$ and record the trajectories every $h/dt$ steps to
get the prediction trajectories $\left( X_{1}^{i,j},\dots
,X_{K}^{i,j}\right) $.

We then calculate the mean trajectory for each ensemble, $\overline{X}%
_{k}^{i}=\frac{1}{N_{ens}}\sum_{j=1}^{N_{ens}}X_{k}^{i,j}$, $k=1,\dots ,K$.
The RMSE measures, in an average sense, the difference between the mean
ensemble trajectory and the true data trajectory:
\begin{equation*}
\mathrm{RMSE}(kh):=\left( \frac{1}{N_{0}}\sum_{i=1}^{N_{0}}\left\vert 
\overline{X}_{k}^{i}-x_{\left( n_{i}+k\right) h}\right\vert ^{2}\right)
^{1/2}.
\end{equation*}
The RMSE measures the accuracy of the mean ensemble prediction; $\mathrm{{RMSE}=0}$ corresponds to a perfect prediction, and small RMSEs are desired.

%%%%%%%%%%%%%%%%%%%%%%
 %\begin{comment}
\begin{figure}[tbp]
\begin{center}
\begin{tabular}{c}
\resizebox{0.95\textwidth}{!}{\includegraphics{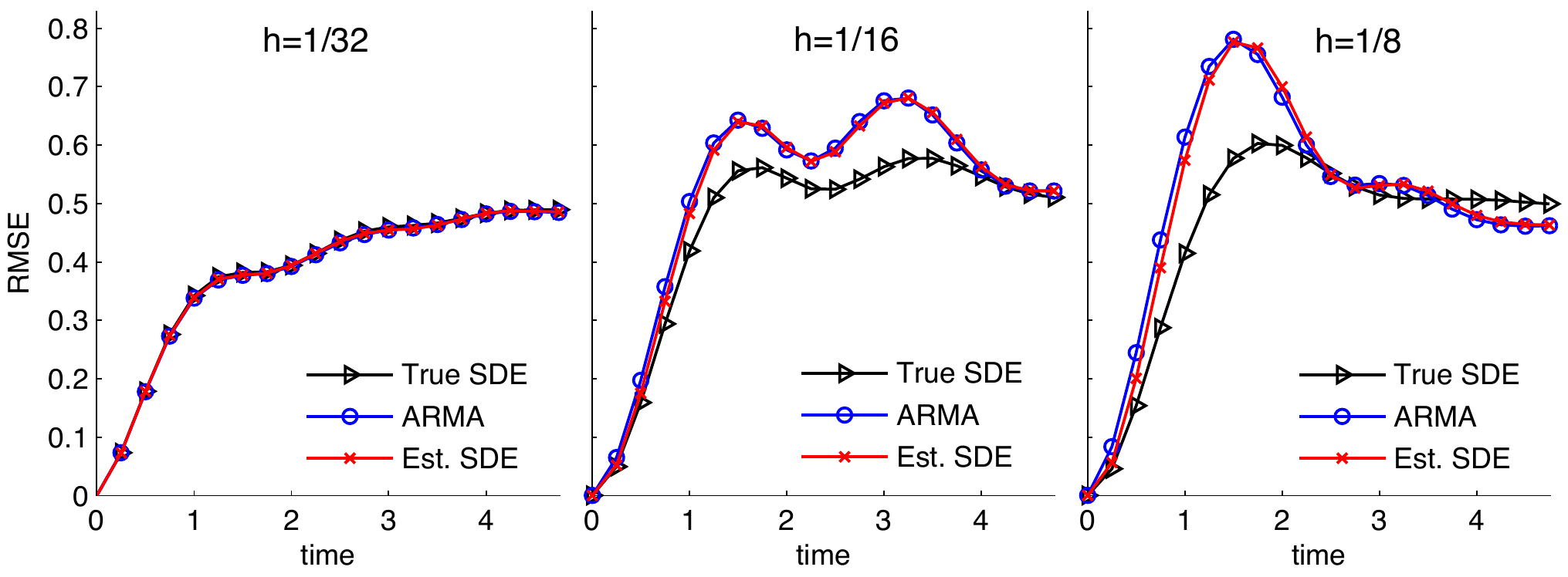}}
\end{tabular}
\end{center}
\vspace{-4mm}
\caption{The linear Langevin system: RMSEs of $10^4$ forecasting ensembles with size $N_{ens}=20$, produced by the true system (denoted by True SDE), the system with estimated parameters (denoted by Est.~SDE), and
the ARMA model. }
\label{fig:RMSE_OU}
\end{figure}
 %\end{comment}
%%%%%%%%%%%%%%%%%%%%%%%%%%%%%%

The computed RMSEs for $N_{0}=10^{4}$ ensembles with $N_{ens}=20$ are shown in
Figure~\ref{fig:RMSE_OU}. The ARMA$\left( 2,1\right) $ model reproduces almost exactly
the RMSEs of the true system for all three observation
step-sizes, while the estimated system has RMSEs deviating from that of the true system as $h$ increases. The estimated system has smaller RMSEs than the true system, because it underestimates the variance of the true process $x_t$ (that is, $\frac{\hat{\sigma}^2}{2\hat{\alpha}\hat{\gamma}} < \frac{\sigma^2}{2\alpha\gamma}$) and because the means of $x_t$ decay exponential to zero. The steady increase
in RMSE, even for the true system, is entirely expected because the
forecasting ensemble is driven by independent realizations of the
forcing, as one cannot infer the white noise driving the system that
originally generated the data.

\subsection{The Kramers oscillator}

We consider the Kramers equation in the following form%
\begin{eqnarray}
dx_{t} &=&y_{t}dt, \notag \\
dy_{t} &=&(-\gamma y_{t}-\beta ^{-2}x_{t}^{3}+x_{t})dt+\sigma dB_{t}, \label{Kramer}
\end{eqnarray}%
for which there are two potential wells located at $x=\pm\beta$.

In the continuous-time approach, the estimator $\hat{\theta}=\left( \hat{\gamma},\hat{\beta},\hat{\sigma}\right) $ is given by 
\begin{equation}
\hat{\theta}=\arg \min_{\theta =\left( \gamma ,\beta ,\sigma \right) }\left[
\sum_{n=1}^{N-3}\frac{3}{2}\frac{\left[ \hat{y}_{n+2}-\hat{y}_{n+1}+h(\gamma 
\hat{y}_{n}+\beta ^{-2}x_{n}^{3}-x_{n})\right] ^{2}}{h\sigma ^{2}}+\left(
N-3\right) \log \sigma ^{2}\right] .  \label{MLE_Kramers}
\end{equation}%

As for the linear Langevin system case, we begin by investigating the
reliability of the estimators. A hundred simulated data sets are
generated from the above Kramers oscillator with true parameters
$\gamma=0.5, \beta = 1/\sqrt{10}, \sigma =1$, and with initial condition $x_0=y_0=1/2$ and integration time interval
$[0,10^4]$. The estimators of $(\gamma, \beta,\sigma)$ are computed for
each data set.  Empirical mean and standard deviation of the estimators
are shown in Table \ref{tab:ST}. We observe that the biases in the
estimators increase as $h$ increases, in particular, the estimator of
$\hat{\gamma}$ has a very large bias.

%%%%%%%%%%%%%%%%%%% Kramers, corresponding to M2
\begin{table}[tbp]
\caption{Mean and standard deviation of the estimators of the parameters $(%
\protect\gamma ,\protect\beta ,\protect\sigma )$ of the Kramers equation in
the continuous-time approach, computed on 100 simulations. }
\label{tab:ST}\centering%
\begin{tabular}{clccc}
\hline
Estimator & True value & $h=1/32$ & $h=1/16$ & $h=1/8$ \\ \hline
$\hat{\gamma}$ & $0.5$ & $0.8726~(0.0063)$ & $1.2049~\left( 0.0057\right) $
& $1.7003~\left( 0.0088\right) $ \\ 
$\hat{\beta}$ & $0.3162$ & $0.3501~\left( 0.0007\right) $ & $0.3662~\left(
0.0007\right) $ & $0.4225~\left( 0.0009\right) $ \\ 
$\hat{\sigma}$ & $1$ & $0.9964~\left( 0.0014\right) $ & $1.0132~\left(
0.0027\right) $ & $1.1150~\left( 0.0065\right) $ \\ \hline
\end{tabular}%
\end{table}
%%%%%%%%%%%%%%%%%%%%%%%%%%%%

%%%%%%%%%%%%%%%%%%%%%%%%%%%%
%%%%%%%%%%%%%% RMSE compare
%\begin{comment}
\begin{figure}[h]
\begin{center}
\begin{tabular}{c}
\resizebox{0.60\textwidth}{!}{\includegraphics{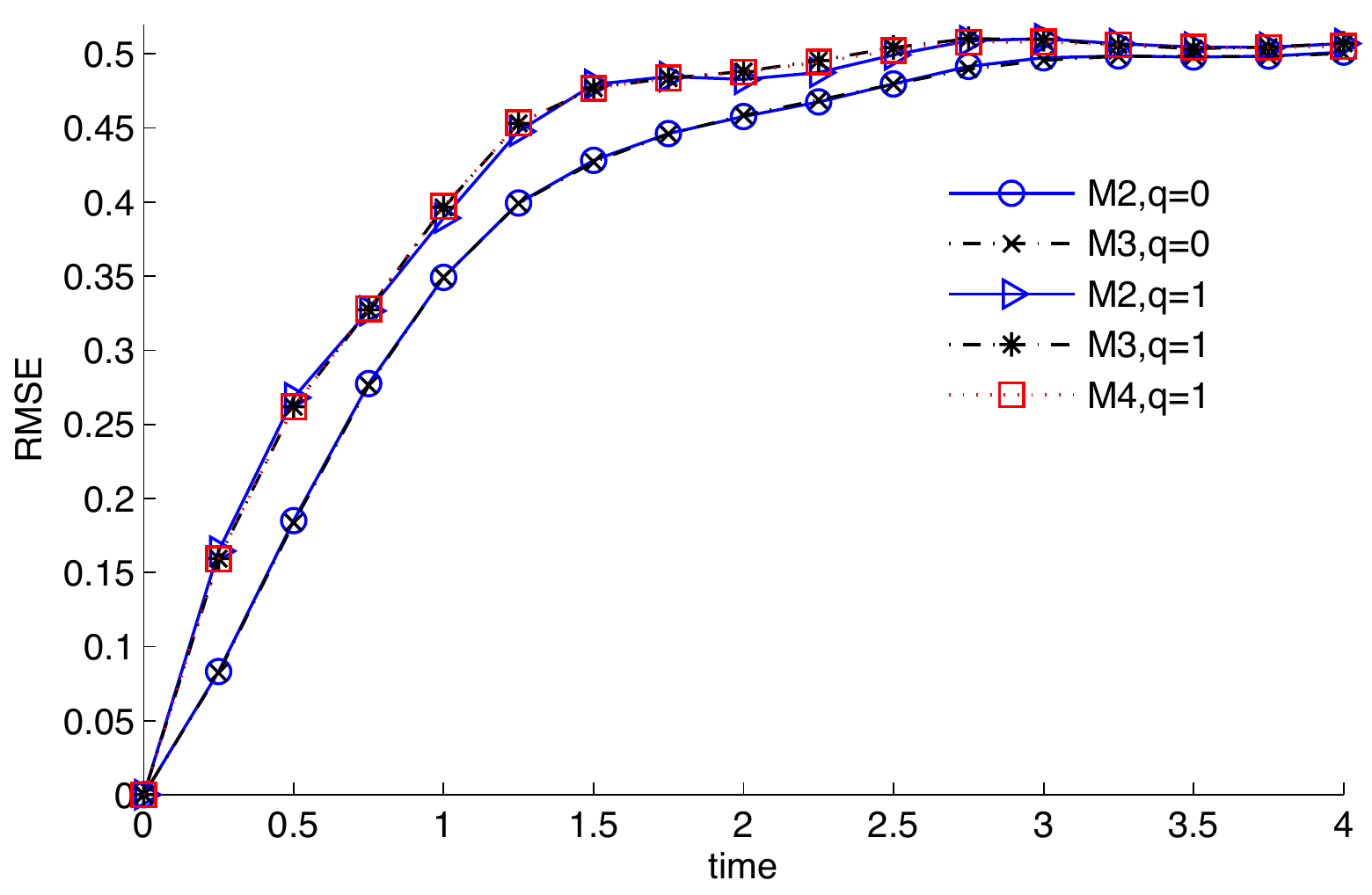}}  
\end{tabular}%
\end{center}
\vspace{-6mm}
\caption{RMSEs of model (M2), (M3), (M4) with ensemble size $N_{ens}=20$ in the case $h=1/8$. Models with $q=1$ have larger RMSEs than the models with $q=0$. In the case $q=0$, models (M2) and (M3) have almost the same RMSEs. }
\label{fig:M234}
\end{figure}
%\end{comment}
%%%%%%%%%%%%%%%%%%%%%

For the discrete-time approach, we have to select one of the four
NARMA$\left( 2,q\right) $ models, Model (M1)--(M4). We make the
selection using data only from a single long trajectory
(e.g. from the time interval $[0,T]$ with $T=2^{18}\approx 2\times
10^5$), and we use the first half of the data to estimate the parameters.
We first estimate the parameters for each NARMA model with $q=0$ and
$q=1$, using the conditional likelihood method described in
Section~\ref{sec_Narma}. Then we make a selection by the criteria
proposed in Section~\ref{sec_Narma}. First, we test numerical stability
by running the model for a large time for different realizations of the
noise sequence.  We find that for our model, using the values of $h$
tested here, Model (M1) is often numerically unstable,
so we do not compare it to the other schemes here. (In situations where
the Euler scheme is more stable, e.g., for smaller values of $h$ or for
other models, we would expect it to be useful as the basis of a NARMA
approximation.) Next, we test the performance of each of the models (M2)--(M4). The
RMSEs of models (M2), (M3) with $q=0$ and $q=1$ and Model (M4) with
$q=1$ are shown in Figure~\ref{fig:M234}. In the case $q=1$, the RMSEs
for models (M2)--(M4) are very close, but they are larger than the RMSEs
of models (M2) and (M3) with $q=0$. To make further selection between
models (M2) and (M3) with $q=0$, we test their reproduction of the
long-term statistics.  Figure \ref{m2m3} shows that model (M3)
reproduces the ACFs and PDFs better than model (M2), hence model (M3)
with $q=0$ is selected.
%%%%%%%%%%  model selection: pdf and act
%\begin{comment}
\begin{figure}[tbp]
\begin{center}
 \begin{tabular}{cc}
  \resizebox{0.45\textwidth}{!}{\includegraphics{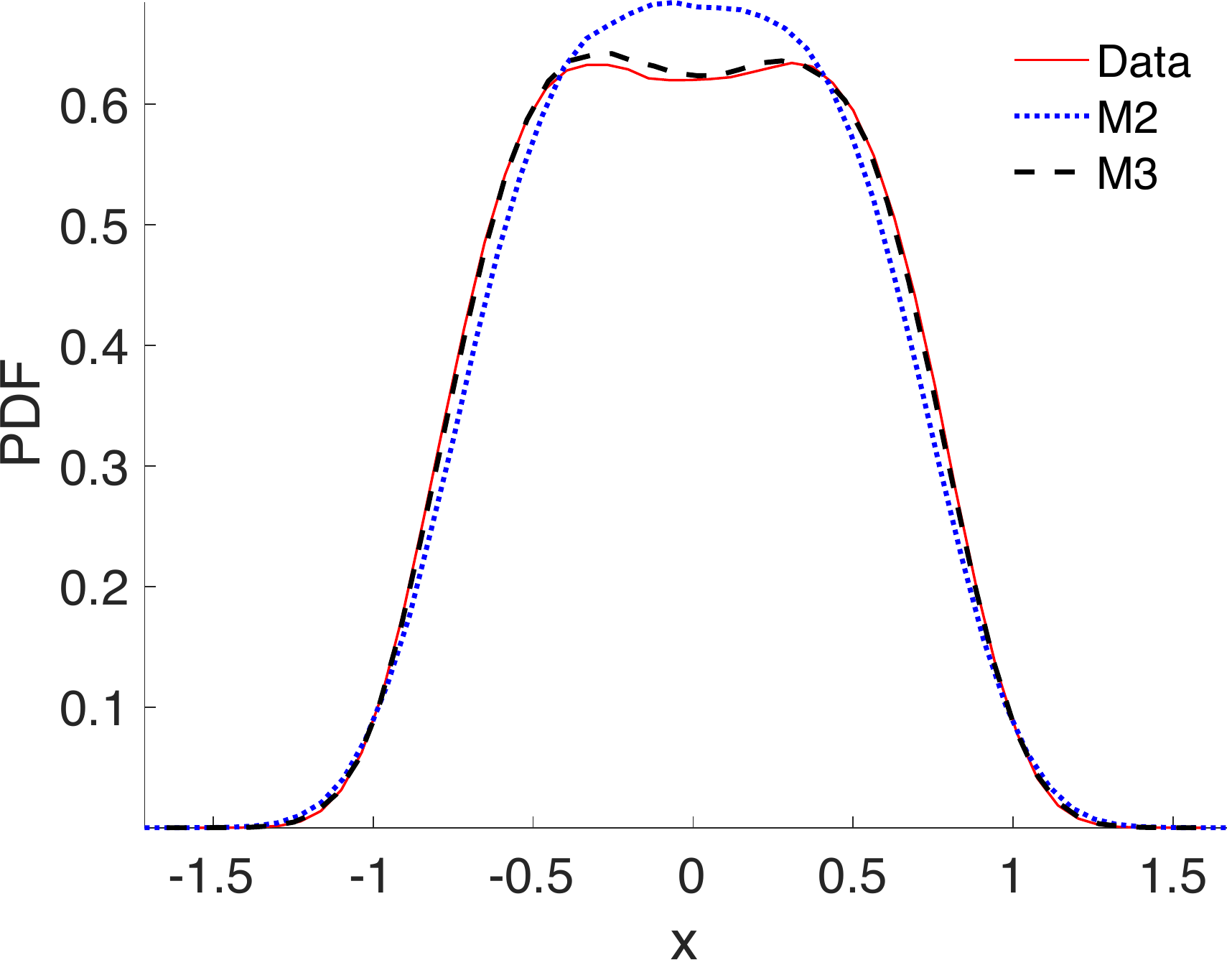}} &
 \resizebox{0.45\textwidth}{!}{\includegraphics{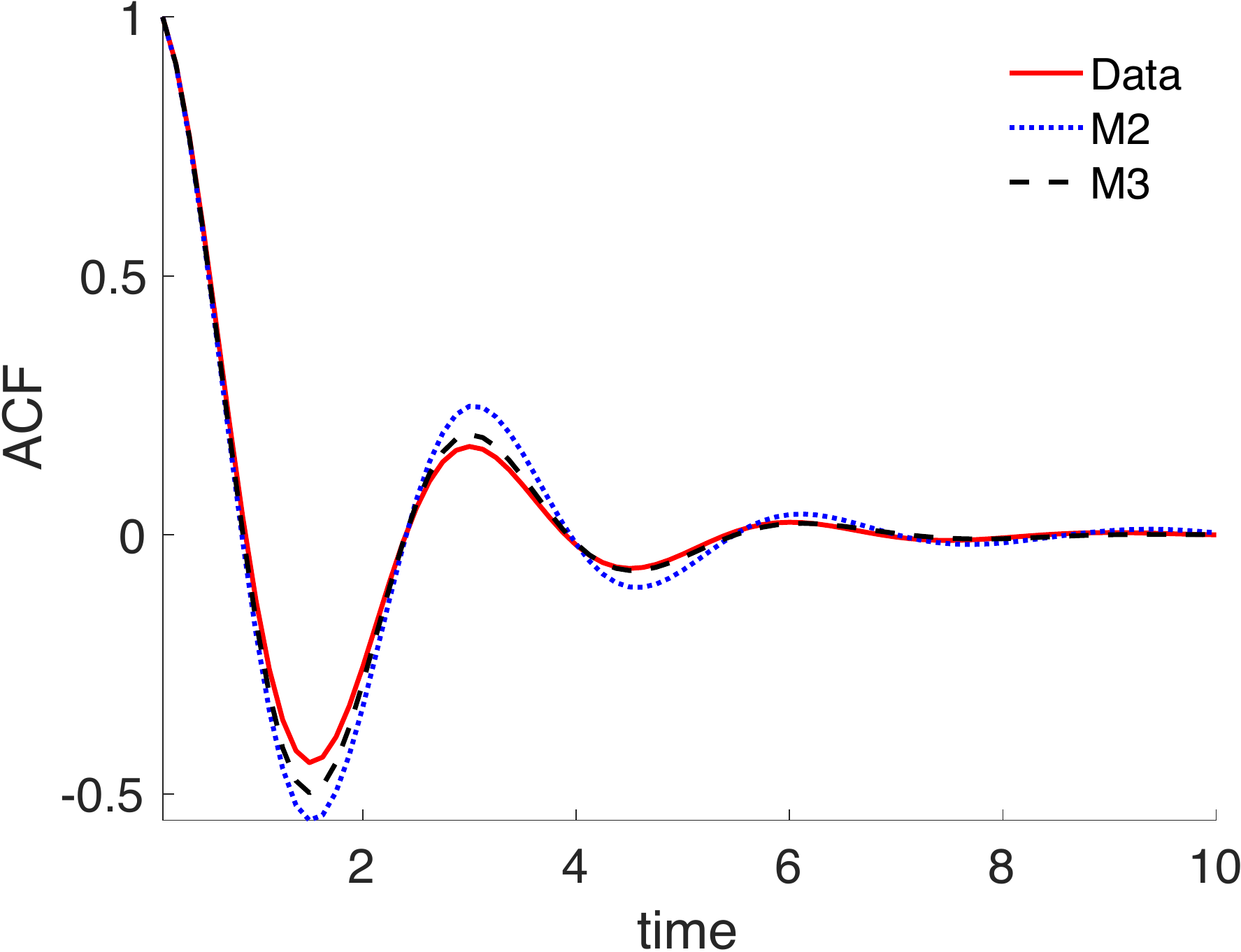}}
  \end{tabular}
\end{center}
\vspace{-6mm}
\caption{Empirical PDFs and ACFs of the NARMA model (M2), (M3) and data in the case $h=1/8$. Model (M3) reproduces the ACF and PDF better than model (M2). }
\label{m2m3}
\end{figure}
%\end{comment}
%%%%%%%%%%%%%%%%%%%%%%

The mean and standard deviation of the estimated parameters of model
(M3) with $q=0$ and 100 simulations are shown in Table
\ref{tab:M3}. Unlike in the linear Langevin system case, we do not have
a theoretical value for these parameters. However, note that when
$h=1/32$, $\hat{a}_1$ and $\hat{a}_2$ are close to $2-\gamma h +0.5\gamma^2
h^2 =1.9845$ and $-(1-\gamma h +0.5\gamma^2 h^2)=-0.9845$ respectively,
which are the coefficients in Eq.~(\ref{IT2_narma}) from the
\Ito-Taylor scheme. This indicates that when $h$ is small, the NARMA
model is close to the numerical scheme, because both the NARMA and the
numerical scheme approximate the true system well. On the other hand,
note that $\hat{\sigma}_W$ does not increase monotonically as $h$
increases. This clearly distinguishes the NARMA model from the numerical
schemes.

%%%%%%%%%%% M3 parameter statistics
\begin{table}[tbp]
\caption{Mean and standard deviation of the estimators of the parameters 
%$(a_1, a_2, b_1, b_2, b_3, \mu, \sigma_W)$ 
of the NARMA model (M3) with $q=0$ in
the discrete-time approach, computed from 100 simulations. }
\label{tab:M3}\centering%
\begin{tabular}{cccc} 
\hline
Estimator & $h=1/32$ & $h=1/16$ & $h=1/8$ \\ \hline
$\hat{a}_{1}$ & 1.9906 (0.0004) & 1.9829 (0.0007)  & 
1.9696 (0.0014) \\ 
$-\hat{a}_{2}$ & 0.9896(0.0004) & 0.9792 (0.0007) & 0.9562 (0.0014) \\ 
$-\hat{b}_{1}$ & 0.3388 (0.1572) & 0.6927 (0.0785) & 1.2988 (0.0389) \\ 
$\hat{b}_{2}$  & 0.0300 (0.1572) & 0.0864 (0.0785) & 0.1462 (0.0386) \\ 
$\hat{b}_{3}$  & 0.0307 (0.1569) & 0.0887 (0.0777) & 0.1655 (0.0372) \\ 
$-\hat{\mu} ~\left( \times 10^{-5}\right) $ & 0.0377 (0.0000) & 0.1478 (0.0000) & 
0.5469 (0.0001) \\ 
$\hat{\sigma}_{W}$ & 0.0045 (0.0000) & 0.1119 (0.0001) & 0.0012 (0.0000)\\ \hline
\end{tabular}\\
\end{table}
%%%%%%%%%%%%%%%%%%%%%%%%%%%

Next, we compare the performance of the NARMA model and the estimated
Kramers system in reproducing long-term statistics and predicting
short-term dynamics.
The empirical PDFs and ACFs are shown in Figure
\ref{fig:acf_Kramers}. The NARMA models can reproduce the PDFs and ACFs
equally well for three cases. The estimated Kramers system amplifies the
depth of double wells in the PDFs, and it misses the oscillation of the
ACFs.
%%%%%%%%%%%%%%%%%%%%
%\begin{comment}
\begin{figure}[tbp]
\begin{center}
\begin{tabular}{c}
\resizebox{0.90\textwidth}{!}{\includegraphics{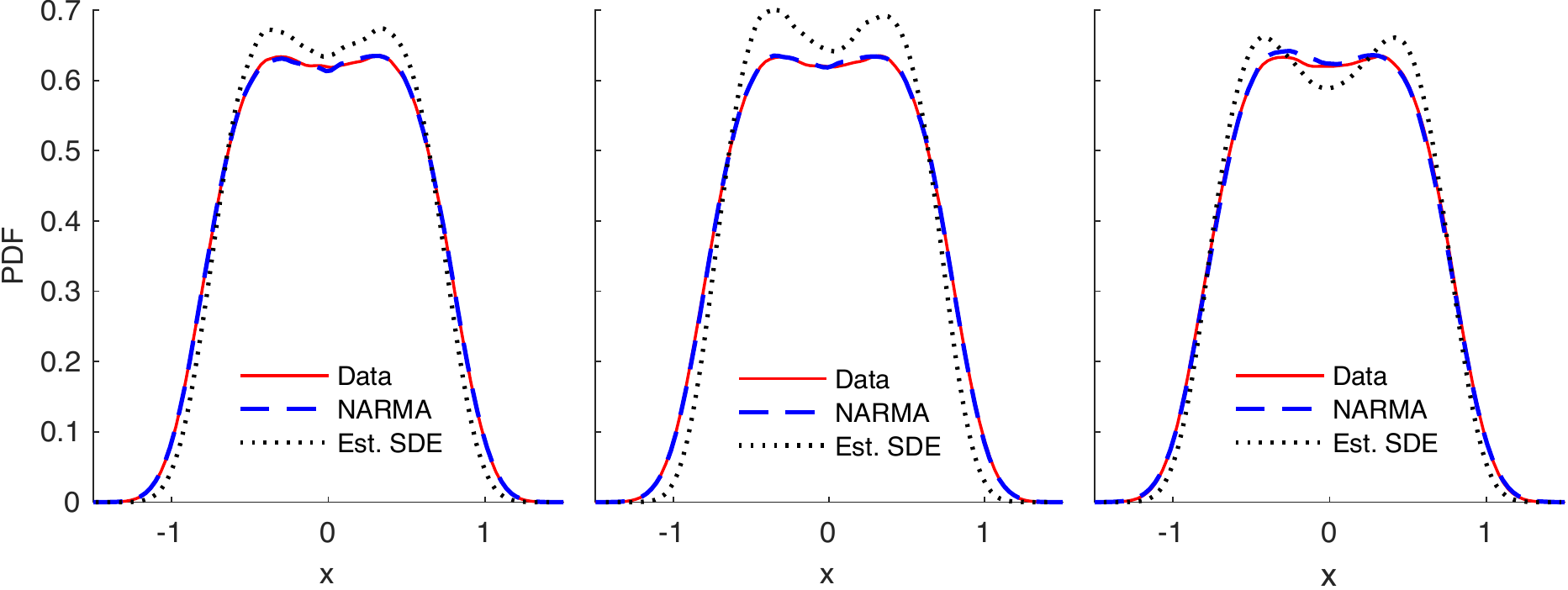}}  \\
\resizebox{0.90\textwidth}{!}{\includegraphics{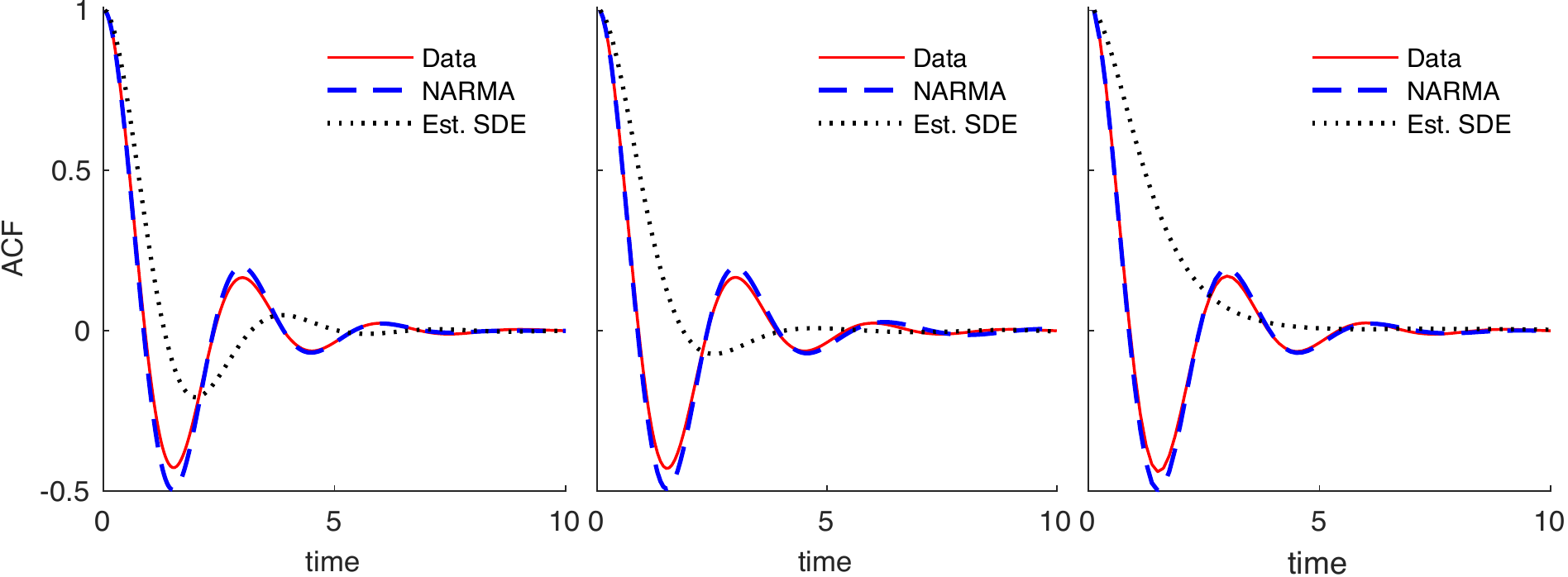}}  
\end{tabular}
\end{center}
\vspace{-6mm}
\caption{Empirical PDFs and ACFs of the NARMA model (M3) with $q=0$ and the estimated Kramers system, in the cases $h=1/32$ (left column), $h=1/16$  (middle column) and $h=1/8$  (right column). These statistics are better reproduced by the NARMA models than by the estimated Kramers systems.}
\label{fig:acf_Kramers}
\end{figure}
%\end{comment}
%%%%%%%%%%%%%%%%%%%%%%

Results for RMSEs for $N_{0}=10^{4}$ ensembles with size $N_{ens}=20$ are shown in
Figure~\ref{fig:M2}. The
NARMA model reproduces almost exactly the RMSEs of the true Kramers
system for all three step-sizes, while the estimated Kramers system has
increasing error as $h$ increases, due to the increasing biases in the
estimators.

%%%%%%%%%%%%%%%%%%%%%%
%\begin{comment}
\begin{figure}[tbp]
\begin{center}
\begin{tabular}{c}
\resizebox{0.90\textwidth}{!}{\includegraphics{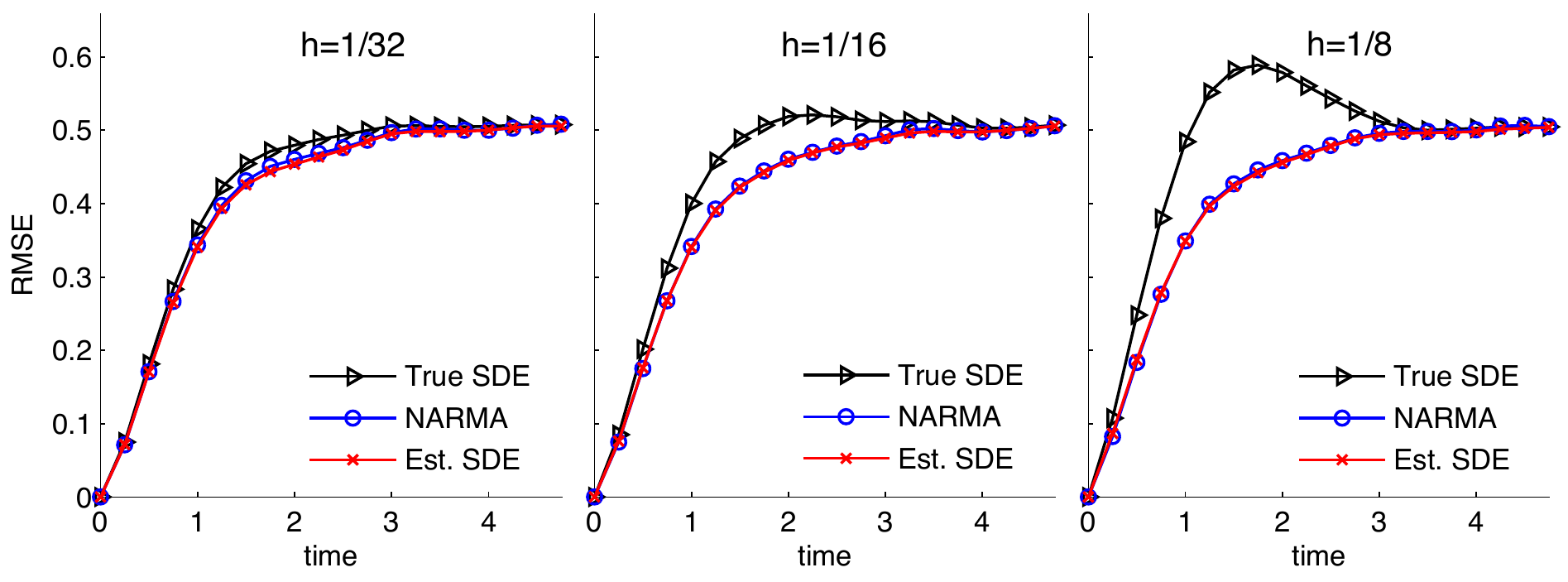}}
\end{tabular}
\end{center}
\vspace{-6mm}
\caption{The Kramers system: RMSEs of $10^4$ forecasting ensembles
 with size $N_{ens}=20$, produced by the true Kramers system, the Kramers system with estimated parameters, and the NARMA model (M3) with $q=0$. The NARMA model has almost the same RMSEs as the true system
  for all the observation spacings, while the estimated system has larger RMSEs.
  }
\label{fig:M2}
\end{figure}
%\end{comment}
%%%%%%%%%%%%%%%%%%%%%%%%%%%%%%

Finally, in Figure~\ref{fig:M2hsmall}, we show some results using
  a much smaller observation spacing, $h=1/1024$.
Figure~\ref{fig:M2hsmall}(a) shows the estimated parameters, for both
the continuous and discrete-time models.  (Here, the discrete-time
model is M2.)  Consistent with the theory in \cite{ST12}, our parameter
estimates for the continuous time model are close to their true
values for this small value of $h$.  Figure~\ref{fig:M2hsmall}(b)
compares the RMSE of the continuous-time and discrete-time models on the
same forecasting task as before.  The continuous-time approach now
performs much better, essentially as well as the true model.  Even in
this regime, however, the discrete-time approach remains competitive.

%%%%%%%%%%%%%%%%%%%%%%%%%%%%%%%%
%\begin{comment}
\begin{figure}[tbp]
  \begin{center}
    \begin{tabular}{cp{0.2in}c}
      \begin{tabular}{lll}
        \hline
        \multicolumn{3}{l}{Continuous-time model parameters}\\
        $\hat{\gamma}$ & $-\hat{\beta}$  & $\hat{\sigma}$ \\
        0.5163                &  0.3435           & 1.0006  \\[2ex]
        \hline
        \multicolumn{3}{l}{Discrete-time model parameters}\\
        $\hat{a}_1$ & $-\hat{a}_2$  & $-\hat{b}_1$ \\
        1.9997         &  0.9997         &   0.0097 \\[2ex]
        $-\hat{b}_2$ & $- \hat{\mu} (\times 10^{-8}) $ & $\hat{\sigma_W} (\times 10^{-10})$ \\
        0.0169      & 2.0388  & 6.2165 \\
        \hline
      \end{tabular}
      && \hspace{-9mm}
      %% \resizebox{3in}{!}{\includegraphics[viewport=0in 0in 7.2in 5.7in]{RMSE_M2g-10_2011}}
      \resizebox{3in}{!}{\includegraphics[viewport=0in 2in 5.8in 5.7in]{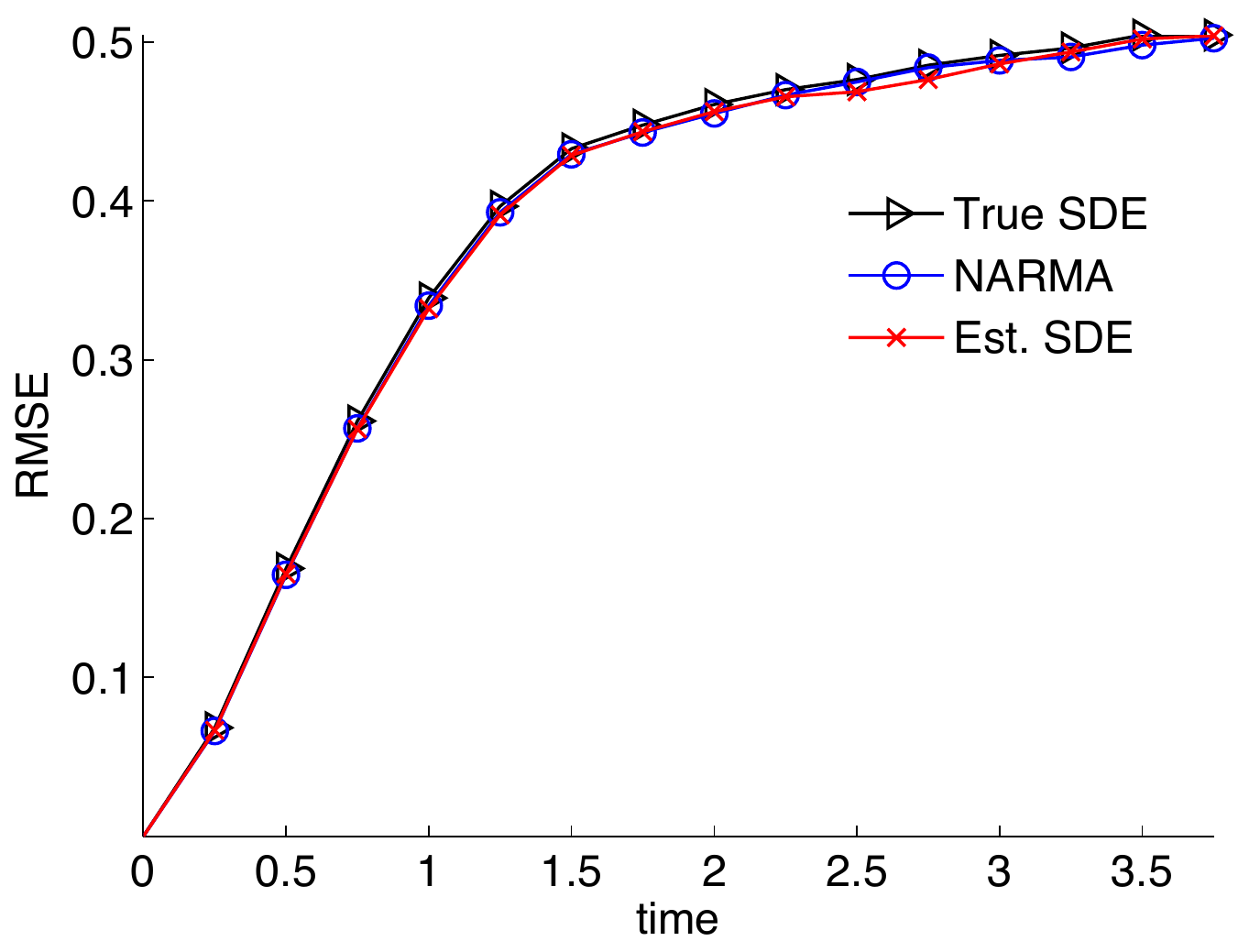}}\\[1.0in]
      (a) Estimated parameter values && (b) $h=1/1024$\\
\end{tabular}
\end{center}
\vspace{-4mm}
\caption{(a) Estimated parameters for the continuous-time and discrete-time models. (b) RMSEs of $10^3$ forecasting ensembles with size $N_{ens}=20$, produced by the true Kramers system (True SDE), the Kramers system with estimated parameters(Est.~SDE), and the NARMA model (M2) with $q=0$. Since $h=1/1024$ is relatively small, the NARMA model and the estimated system have
almost the same RMSEs as the true system. Here the data is generated
by the \Ito-Taylor solver with step size $dt=2^{-15}\approx 3\times
10^{-5}$, and data length is $N=2^{22}\approx 4\times 10^6$. 
%In the forecasting test, the estimated SDE system and the true system are solved by \Ito-Taylor with step size \knew{$dt=1/4096$}.
}
\label{fig:M2hsmall}
\end{figure}% 
%\end{comment}
%%%%%%%%%%%%%%%%%%%%%%

%% \subsection{Discussion of structure design} \label{secDiscussion}
\subsection{Criteria for structure design} \label{secDiscussion}

%%%%%%%%%%%%%%%%%%% Consistency test M2 vs M3
\begin{table}[tbp]
\caption{Consistency test. Values of the estimators in the
  NARMA models (M2) and (M3) with $q=0$. The data come from a long
  trajectory with observation spacing $h=1/32$. Here $N=2^{22}\approx 4\times10^6$. As
  the length of data increases, the estimators of model (M2) have much
  smaller oscillation than the estimators of model (M3).}
\label{tab:ConsTest}\centering%
\begin{tabular}{ccc|cccc}
\hline
Data length& \multicolumn{2}{c|}{Model (M2)} & \multicolumn{3}{c}{Model (M3)} \\ 
$(\times N)$ & $-\hat{b}_1$ & $-\hat{b}_2$  & $-\hat{b}_1$  & $\hat{b}_2$  & $\hat{b}_3$  \\ 
\hline
$1/8$   & 0.3090 & 0.3032   & 0.3622  & 0.0532  & 0.0563   \\ 
$1/4$   & 0.3082 & 0.3049   & 0.3290  & 0.0208  & 0.0217  \\ 
$1/2$   & 0.3088 & 0.3083   & 0.3956  & 0.0868  & 0.0845   \\ 
$1$      & 0.3087 & 0.3054   & 0.3778  & 0.0691  & 0.0697   \\ \hline
\end{tabular}\\
\end{table}
%%%%%%%%%%%%%%%%%

In the above structure selection between model (M2) and (M3), we
followed the criterion of selecting the one that fits the long-term
statistics best. However, there is another practical criterion, namely whether the estimators converge as the
number of samples increases. This is important because the
estimators should converge to the true values of the parameters if the
model is correct, due to the consistency discussed in
Section~\ref{sec_Narma}.  Convergence can be tested by checking the
oscillations of estimators as data length increases: if the oscillations
are large, the estimators are likely not to converge, at least not
quickly. Table \ref{tab:ConsTest} shows the estimators of the
coefficients of the nonlinear terms in model (M2) and (M3), for
different lengths of data. The estimators $\hat{b}_1,\hat{b}_2$ and
$\hat{b}_3$ of model (M3) are unlikely to be convergent, since they vary
a lot for long data sets. On the contrary, the estimators $\hat{b}_1$
and $\hat{b}_2$ of model (M2) have much smaller oscillations, and hence
they are likely to be convergent.

These convergence tests agree with the statistics of the estimators on
100 simulations in Tables~\ref{tab:M3} and~\ref{tab:M2}. Table
\ref{tab:M3} shows that the standard deviations of the estimators
$\hat{b}_1,\hat{b}_2$ and $\hat{b}_3$ of model (M3) are reduced by half
as $h$ doubles, which is the opposite of what is supposed to happen for
an accurate model. On the contrary, Table \ref{tab:M2} shows that the
standard deviations of the parameters of model (M2) increase as $h$
doubles, as is supposed to happen for an accurate model.

In short, model (M3) reproduces better long-term statistics than model
(M2), but the estimators of model (M2) are statistically better (e.g. in
 rate of convergence) than the estimators of
model (M3).  However, the two have
almost the same prediction skill as shown in Figure \ref{fig:M234}, and
both are much better than the continuous-time approach. It is unclear
which model approximates the true process better, and it is likely that
neither of them is optimal. Also, it is unclear which criterion is
better for structure selection: fitting the long-term statistics or
consistency of estimators.  We leave these issues to be addressed in
future work.

%%%%%%%%%%% M2 parameter statistics
\begin{table}[tbp]
\caption{Mean and standard deviation of the estimators of the parameters $(
a_1, a_2, b_1, b_2, \mu, \sigma_W)$ of the NARMA model (M2) with $q=0$ in
the discrete-time approach, computed on 100 simulations. }
\label{tab:M2}\centering%
\begin{tabular}{cccc}
\hline
Estimator & $h=1/32$ & $h=1/16$ & $h=1/8$ \\ \hline
$\hat{a}_{1}$ & 1.9905 (0.0003) & 1.9820 (0.0007)  & 
1.9567 (0.0013) \\ 
$-\hat{a_{2}}$ & 0.9896 (0.0003) & 0.9788 (0.0007) & 0.9508 (0.0014) \\ 
$-\hat{b_{1}}$ & 0.3088 (0.0021) & 0.6058 (0.0040) & 1.1362 (0.0079) \\ 
$-\hat{b_{2}}$ & 0.3067 (0.0134) & 0.5847 (0.0139) & 0.9884 (0.0144) \\ 
$-\hat{\mu} ~\left( \times 10^{-5}\right) $ & 0.0340 (0.0000) & 0.1193 (0.0000) & 
0.2620 (0.0001) \\ 
$\hat{\sigma}_{W}$ & 0.0045 (0.0000) & 0.1119 (0.0001) & 0.0012 (0.0000)\\ 
\hline
\end{tabular}%
\end{table}
%%%%%%%%%%%%%%%%%%%%%%%%%%%

\section{Concluding discussion} \label{sec5}

We have compared a discrete-time approach and a continuous-time approach
to the data-based stochastic parametrization of a dynamical system, in a
situation where the data are known to have been generated by
hypoelliptic stochastic system of a given form. In the continuous time
case, we first estimated the coefficients in the given equations using
the data, and then solved the resulting differential equations; in the
discrete-time model, we chose structures with terms suggested by
numerical algorithms for solving the equations of the given form, with
coefficients estimated using the data.

As discussed in our earlier papers \cite{CL15,LLC15}, the discrete-time
approach has several a priori advantages:
\begin{enumerate} 
\item the inverse problem of estimating the parameters in a model from
  discrete data is in general better-posed in a discrete-time than in a
  continuous-time model. In particular, the discrete time representation
  is more tolerant of relatively large observation spacings.
\item once the discrete-time parametrization has been derived, it can be
  used directly in numerical computation, there is no need of further
  approximation. This is not a major issue in the present paper where
  the equations are relatively simple, but we expect it to grow in
  significance as the size of problems increases.
\end{enumerate}
Our example validates the first of these points; the discrete-time
approximations generally have better prediction skills than the
continuous-time parametrization, especially when the observation spacing
is relatively large. This was also the main source of error in the
continuous models discussed in \cite{CL15}; note that the method for
parameter estimation in that earlier paper was completely different.
Our discrete-time models also have better numerical properties, e.g.,
when all else is equal, they are more stable and produce more
  accurate long term statistics than their continuous-time
counterparts.

We expect the advantages of the discrete-time approach to become more
marked as one proceeds to analyze systems of growing complexity,
particularly larger, more chaotic dynamical systems. A number of
  questions remain, first and foremost being the identification of
effective structures; this is of course a special case of the difficulty
in identifying effective bases in the statistical modeling of complex
phenomena.  In the present paper we introduced the idea of using terms
derived from numerical approximations; different ideas were introduced
in our earlier work \cite{LLC15}.  More work is needed to generate
general tools for structure determination.

Another challenge is that, even when one has derived a small number
  of potential structures, we currently do not have a systematic way to
  identify the most effective model.  Thus, the selection of a suitable
  discrete-time model can be labor-intensive, especially compared to the
  continuous-time approach in situations where a parametric family
  containing the true model (or a good approximation thereof) is known.
 On the other hand, continuous-time approaches, in situations where no
  good family of models is known, would face similar difficulties.

Finally, another open question is whether discrete-time approaches
  generally produce more accurate predictions than continuous-time
  approaches for strongly chaotic systems.  Previous work has suggested
  that the answer may be yes.  We plan to address this question more
  systematically in future work.

\paragraph{Acknowledgments.} We would like to thank the anonymous referee and Dr.~Robert Saye for helpful suggestions. 
KKL is supported in part by the National Science Foundation under grant DMS-1418775. AJC and FL are supported in part by the Director, Office of Science, Computational and Technology Research, U.S. Department of Energy, under Contract No. DE-AC02-05CH11231, and by the National Science Foundation under grant DMS-1419044.

%%%%%%%%%%%%%%%%%%%%%%
\appendix

\section{Solutions to the linear Langevin equation\label{solutionLLE}}

Denoting 
\begin{equation*}
\mathbf{X}_{t}\mathbf{=}\left( 
\begin{array}{c}
x_{t} \\ 
y_{t}%
\end{array}%
\right) ,\mathbf{A}=\left( 
\begin{array}{cc}
0 & 1 \\ 
-\alpha & -\gamma%
\end{array}%
\right) ,\text{ }\mathbf{e}=\left( 
\begin{array}{c}
0 \\ 
\sigma%
\end{array}%
\right) ,
\end{equation*}%
we can write Eq.~(\ref{OU}) as 
\begin{equation*}
d\mathbf{X}_{t}\mathbf{=AX}_{t}dt+\mathbf{e}dB_{t}.
\end{equation*}%
Its solution is 
\begin{equation*}
\mathbf{X}_{t}=e^{\mathbf{A}t}\mathbf{X}_{0}+\int_{0}^{t}e^{\mathbf{A}(t-u)}%
\mathbf{e}dB_{u}.
\end{equation*}

The solution at discrete times can be written as%
\begin{eqnarray*}
x_{\left( n+1\right) h} &=&a_{11}x_{nh}+a_{12}y_{nh}+W_{n+1,1}, \\
y_{\left( n+1\right) h} &=&a_{21}x_{nh}+a_{22}y_{nh}+W_{n+1,2},
\end{eqnarray*}%
where $a_{ij}=\left( e^{\mathbf{A}h}\right) _{ij}$ for $i,j=1,2$, and 
\begin{equation}
W_{n+1,i}=\sigma \int_{0}^{h}a_{i2}(u)dB\left( nh+u\right)  \label{Vi}
\end{equation}%
with $a_{i2}(u)=\left( e^{\mathbf{A}(h-u)}\right) _{i2}$ for $i=1,2$. Note
that if $a_{12}\neq 0$, then from the first equation we get $y_{nh}=\left(
x_{\left( n+1\right) h}-a_{11}x_{nh}-V_{n+1,1}\right) /a_{12}$. Substituting
it into the second equation we obtain
\begin{eqnarray*}
x_{\left( n+2\right) h} &=&\left( a_{11}+a_{22}\right) x_{\left( n+1\right)
h}+\left( a_{12}a_{21}-a_{11}a_{22}\right) x_{nh} \\
&&-a_{22}W_{n+1,1}+a_{12}W_{n+1,2}+W_{n+2,1}.
\end{eqnarray*}%
Combining with the fact that $a_{11}+a_{22}=\mathrm{trace}(e^{\mathbf{A}h})$
and $a_{12}a_{21}-a_{11}a_{22}=-e^{-\gamma h}$, we
have 
\begin{equation}
x_{\left( n+2\right) h}=\mathrm{trace}(e^{\mathbf{A}h})x_{\left( n+1\right)
h}-e^{-\gamma h}x_{nh}-a_{22}W_{n+1,1}+W_{n+2,1}+a_{12}W_{n+1,2}.
\label{OU_ts}
\end{equation}

Clearly, the process $\left\{ x_{nh}\right\} $ is a centered Gaussian
process, and its distribution is determined by its autocovariance function.
Conditionally on $\mathbf{X}_{0}$, the distribution of $\mathbf{X}_{t}$ is $\mathcal{N}(e^{\mathbf{A}t}\mathbf{X}_{0},\mathbf{\Sigma }(t))$, where $\mathbf{\Sigma }(t):=\int_{0}^{t}e^{\mathbf{A}u}\mathbf{ee}^{T}e^{\mathbf{A}%
^{T}u}du$. Since $\alpha ,\gamma >0$, the real parts of the eigenvalues of
the $A$, denoted by $\lambda _{1}$ and $\lambda _{2}$, are negative. The
stationary distribution is $\mathcal{N}(0,\mathbf{\Sigma }(\infty ))$, where 
$\mathbf{\Sigma }(\infty )=\lim_{t\rightarrow \infty }\mathbf{\Sigma }(t)$.
If $\mathbf{X}_{0}$ has distribution $\mathcal{N}(0,\mathbf{\Sigma }(\infty
))$, then the process $\left( \mathbf{X}_{t}\right) $ is stationary, and so
is the observed process $\left\{ x_{nh}\right\} $. The following lemma
computes the autocorrelation function of the stationary process $\left\{
x_{nh}\right\} $.

\begin{lemma}
\label{eigen}Assume that the system $(\ref{OU})$ is stationary. Denote by $\left\{ \gamma _{j}\right\} _{j=1}^{\infty }$ the autocovariance function of
the stationary process $\left\{ x_{nh}\right\} $, i.e. $\gamma _{j}:=\mathbb{E}[x_{kh}x_{(k+j)h}] $
for $j\geq 0$. Then $\gamma _{0}=\frac{\sigma ^{2}}{2\alpha \gamma }$, and $\gamma _{j}$ can be represented as 
\begin{equation*} \label{gamma}
\gamma _{j}=\gamma _{0}\times \left\{ 
\begin{array}{cc}
\frac{1}{\lambda _{1}-\lambda _{2}}(\lambda _{1}e^{\lambda _{2}jh}-\lambda
_{2}e^{\lambda _{1}jh}),~ & \text{if }\gamma ^{2}-4\alpha \neq 0; \\ 
e^{\lambda _{0}jh}(1-\lambda _{0}jh), & \text{if }\gamma ^{2}-4\alpha =0%
\end{array}%
\right.
\end{equation*}%
for all $j\geq 0$, where $\lambda _{1}$ and $\lambda _{2}$ are the different
solutions to $\lambda ^{2}+\gamma \lambda +\alpha =0$ when $\gamma
^{2}-4\alpha \neq 0$, and $\lambda _{0}=-\gamma /2$.
\end{lemma}
\begin{proof}
Let $\mathbf{\Gamma }(j):=\mathbb{E}[\mathbf{X}_{kh}\mathbf{X}_{(k+j)h}^{T}]=%
\mathbf{\Sigma }(\infty )e^{\mathbf{A}^{T}jh}$ for $j\geq 0$. Note that $\gamma _{j}=\mathbf{\Gamma }_{11}(j)$, i.e., $\gamma _{j}$ is the first
element of the matrix $\mathbf{\Gamma }(j)$. Then it follows that 
\begin{equation*}
\gamma _{0}=\mathbf{\Sigma }_{11}(\infty ),~\gamma _{j}=\left( \mathbf{\Sigma }(\infty )e^{\mathbf{A}^{T}jh}\right) _{11}.
\end{equation*}%
If $\gamma ^{2}-4\alpha \neq 0$, then $\mathbf{A}$ has two different
eigenvalues $\lambda _{1}$ and $\lambda _{2}$, and it can be written as 
\begin{equation*}
\mathbf{A}=\mathbf{Q\Lambda Q}^{-1}~\text{with}~\mathbf{Q=}\left( 
\begin{array}{cc}
1 & 1 \\ 
\lambda _{1} & \lambda _{2}%
\end{array}%
\right) ,\mathbf{\Lambda }=\left( 
\begin{array}{cc}
\lambda _{1} & 0 \\ 
0 & \lambda _{2}%
\end{array}%
\right) .
\end{equation*}%
The covariance matrix $\mathbf{\Sigma }(\infty )$\ can be computed as 
\begin{equation}
\mathbf{\Sigma }(\infty )=\lim_{t\rightarrow \infty }\int_{0}^{t}\mathbf{Q}%
e^{\mathbf{\Lambda }u}\mathbf{Q}^{-1}\mathbf{ee}^{T}\mathbf{Q}^{-T}e^{\mathbf{\Lambda }^{T}u}\mathbf{Q}^{T}du=\sigma ^{2}\left( 
\begin{array}{cc}
\frac{1}{2ab} & 0 \\ 
0 & -\frac{1}{2b}%
\end{array}%
\right) .  \label{Sigma}
\end{equation}%
This gives $\gamma _{0}=\mathbf{\Sigma }_{11}(\infty )=\frac{\sigma ^{2}}{2\gamma \alpha }$ and for $j>0$,
\begin{equation*}
\gamma _{j}=\mathbf{\Sigma }_{11}(\infty )\left( e^{\mathbf{A}^{T}jh}\right)
_{11}=\frac{1}{\lambda _{1}-\lambda _{2}}(\lambda _{1}e^{\lambda
_{2}jh}-\lambda _{2}e^{\lambda _{1}jh})\gamma (0).
\end{equation*}%

In the case $\gamma ^{2}-4\alpha =0$, $\mathbf{A}$ has a single eigenvalue $\lambda _{0}=-\frac{\gamma }{2}$, and it can be transformed to a Jordan block%
\begin{equation*}
\mathbf{A=Q\Lambda Q}^{-1}~\text{with}~\mathbf{Q=}\left( 
\begin{array}{cc}
1~ & 0 \\ 
\lambda _{0}~ & 1%
\end{array}%
\right) ,\mathbf{\Lambda }=\left( 
\begin{array}{cc}
\lambda _{0}~ & 1 \\ 
0~ & \lambda _{0}%
\end{array}%
\right) .
\end{equation*}%
This leads to the same $\mathbf{\Sigma }(\infty )$ as in (\ref{Sigma}).
Similarly, we have $\gamma _{0}=\frac{\sigma ^{2}}{2\gamma \alpha }$ and 
\begin{equation*}
\gamma _{j}=\mathbf{\Sigma }_{11}(\infty )\left( e^{\mathbf{A}^{T}jh}\right)
_{11}=e^{\lambda _{0}jh}(1-\lambda _{0}jh)\gamma _{0}.
\end{equation*}
\vspace{-2mm}
\end{proof}

\section{ARMA\ processes}
\label{section:arma}

We review the definition and computation of autocovariance function of ARMA\
processes in this subsection. For more details, we refer to \cite[Section 3.3]%
{BD91}.
\begin{definition}
\label{ARMAdef}The process $\left\{ X_{n},n\in \mathbb{Z}\right\} $ is said
to be an ARMA($p,q$) process if it is stationary process satisfying 
\begin{equation}
X_{n}-\phi _{1}X_{n-1}-\cdots -\phi _{p}X_{n-p}=W_{n}+\theta
_{1}W_{n-1}+\cdots +\theta _{q}W_{n-q},  \label{ARMApq}
\end{equation}%
for every $n$, where $\left\{ W_{n}\right\} $ are i.i.d $\mathcal{N}%
(0,\sigma _{W}^{2})$, and if the polynomials $\phi(z):=1-\phi_1 z - \dots - \phi_p z^p$ and $\theta(z) := 1+ \theta_1 z +\dots +\theta_q z^q$ have no common factors. If $\left\{ X_{n}-\mu \right\} $ is an ARMA($p,q$)
process, then $\left\{ X_{n}\right\} $ is said to be an ARMA($p,q$) process
with mean $\mu $. The process is causal if $\phi(z) \neq 0$ for all $|z|\leq 1$. The process is invertible if $\theta(z) \neq 0$ for all $|z|\leq 1$. 
\end{definition}

The autocovariance function $\left\{ \gamma (k)\right\} _{k=1}^{\infty }$ of
an ARMA($p,q$) can be computed from the following difference equations,
which are obtained by multiplying each side of (\ref{ARMApq}) by $X_{n-k}$
and taking expectations, 
\begin{eqnarray}
\gamma (k)-\phi _{1}\gamma (k-1)-\dots -\phi _{p}\gamma (k-p) &=&\sigma
_{W}^{2}\sum_{k\leq j\leq q}\theta _{j}\psi _{j-k},~~0\leq k< \max\{p,q+1\},
\label{gamma1} \\
\gamma (k)-\phi _{1}\gamma (k-1)-\dots -\phi _{p}\gamma (k-p)
&=&0,~~~~~~~~ k\geq \max\{p, q+1\},  \label{gamma_diff}
\end{eqnarray}%
where $\psi _{j}$ in (\ref{gamma1}) is computed as follows (letting $\theta
_{0}:=1$ and $\theta _{j}=0$ if $j>q$)

\begin{equation*}
\psi _{j}=\left\{ 
\begin{array}{cc}
{\theta _{j}+\sum_{0<k\leq j}\phi _{k}\psi _{j-k}},~ & \text{for }j<\max\{p,q+1\}; \\ 
{\sum_{0<k\leq p}\phi _{k}\psi _{j-k}}, & \text{for }j\geq \max\{p,q+1\}.%
\end{array}%
\right.  % \label{psi_1}
\end{equation*}%
Denote $\left( \zeta _{i},i=1,\dots ,k\right) $ the distinct zeros of $\phi
(z):=1-\phi _{1}z-\cdots -\phi _{p}z^{p}$, and let $r_{i}$ be the
multiplicity of $\zeta _{i}$ (hence $\sum_{i=1}^{k}r_{i}=p$). The general
solution of the difference Eq.~(\ref{gamma_diff}) is 
\begin{equation}
\gamma (n)=\sum_{i=1}^{k}\sum_{j=0}^{r_{i}-1}\beta _{ij}n^{j}\zeta _{i}^{-n}%
\text{, for }n\geq \max\{p,q+1\}-p,  \label{gamma_roots}
\end{equation}%
where the $p$ constants $\beta _{ij}$ (and hence the values of $\gamma (j)$
for $0\leq j<\max\{p,q+1\}-p$ ) are determined from (\ref{gamma1}).

\begin{example}[ARMA($2,0$)]\textbf{.} \label{exmpl20}
{\rm For an ARMA(2,0) process $X_{n}-\phi
_{1}X_{n-1}-\phi _{2}X_{n-2}=W_{n}$, its autocovariance function is%
\begin{equation*}
\gamma (n)=\left\{ 
\begin{array}{cc}
{\beta _{1}\zeta _{1}^{-n}+\beta _{2}\zeta _{2}^{-n}},~ & \text{if }\phi
_{1}^{2}+4\phi _{2}\neq 0; \\ 
{\left( \beta _{1}+\beta _{2}n\right) \zeta ^{-n}}, & \text{if }\phi
_{1}^{2}+4\phi _{2}=0%
\end{array}%
\right.  \label{gammaAR2}
\end{equation*}%
for $n\geq 0$, where $\zeta _{1}$, $\zeta _{2}$ or $\zeta $ are the zeros of 
$\phi (z)=1-\phi _{1}z-\phi _{2}z^{2}$. The constants $\beta _{1}\,$and $\beta _{2}$ are computed from the equations%
\begin{eqnarray*}
\gamma (0)-\phi _{1}\gamma (1)-\phi _{2}\gamma (2) &=&\sigma _{W}^{2}, \\
\gamma (1)-\phi _{1}\gamma (0)-\phi _{2}\gamma (1) &=&0.
\end{eqnarray*}
}
\end{example}

\begin{example}[ARMA($2,1$)]\textbf{.} \label{exmpl21} 
{\rm
For an ARMA(2,1) process $X_{n}-\phi
_{1}X_{n-1}-\phi _{2}X_{n-2}=W_{n}+\theta _{1}W_{n-1}$, we have $\psi
_{0}=1,\psi _{1}=\phi _{1}$. Its autocovariance function is of the same form as that
in $(\ref{gammaAR2})$, where the constants $\beta _{1}\,$and $\beta _{2}$
are computed from the equations 
\begin{eqnarray*}
\gamma (0)-\phi _{1}\gamma (1)-\phi _{2}\gamma (2) &=&\sigma
_{W}^{2}(1+\theta _{1}^{2}+\theta _{1}\phi _{1}), \\
\gamma (1)-\phi _{1}\gamma (0)-\phi _{2}\gamma (1) &=&\sigma _{W}^{2}\theta
_{1}.
\end{eqnarray*}
}
\end{example}

\section{Numerical schemes for hypoelliptic SDEs with additive noise\label{schemeLE}}
Here we briefly review the two numerical schemes, the Euler-Maruyama scheme
and the \Ito-Taylor scheme of strong order 2.0, for hypoelliptic systems
with additive noise%
\begin{eqnarray*}
dx &=&ydt,  \label{GLE} \\
dy &=&a(x,y)dt+\sigma dB_{t},  \notag
\end{eqnarray*}%
where $a:\mathbb{R}^2\to \mathbb{R}$ satisfies suitable conditions so that the system is ergodic.

In the following, the step size of all schemes is $h$, and $W_{n}=\sigma 
\sqrt{h}\xi _{n},~Z_{n}=\sigma h^{3/2}\left( \xi _{n}+\eta_n /\sqrt{3}
\right) $, where $\left\{ \xi _{n}\right\} $ and $\left\{ \eta
_{n}\right\} $ are two i.i.d sequences of $\mathcal{N}(0,1)$ random
variables.  \\
\textbf{Euler-Maruyama }(EM)\textbf{:} \vspace{-2mm}
\begin{eqnarray}
x_{n+1} &=&x_{n}+y_{n}h,  \label{EMxy} \\
y_{n+1} &=&y_{n}+ha(x_{n},y_{n})+W_{n+1}.  \notag
\end{eqnarray}
\textbf{\Ito-Taylor scheme of strong order 2.0 }(IT2)\textbf{:}%
\begin{eqnarray}
x_{n+1} &=&x_{n}+hy_{n}+ 0.5h^{2}a\left( x_{n},y_{n}\right) +Z_{n+1},  \notag \\
y_{n+1} &=&y_{n}+ha\left( x_{n},y_{n}\right) + 0.5h^{2}\left[
a_{x}(x_{n},y_{n})y_{n}+\left( aa_{y}+ 0.5\sigma ^{2}a_{yy}\right)
(x_{n},y_{n})\right]  \label{IT2xy} \\
&&+W_{n+1}+a_{y}(x_{n},y_{n})Z_{n+1}+a_{yy}(x_{n},y_{n})\sigma ^{2}\frac{h}{6}(W_{n+1}^2-h).
\notag
\end{eqnarray}
%where  $I_{110}^{n+1}=\int_{t_{n}}^{t_{n+1}}\int_{t_{n}}^{t}\left(B_{s}-B_{t_{n}}\right) dB_{s}dt$. 

The \Ito-Taylor scheme of order 2.0 can be derived as follows (see e.g.
Kloeden and Platen \cite{Hu96, KP99} ). The differential equation can be rewritten
in the integral form:\vspace{-2mm}
\begin{eqnarray*}
x_{t} &=&x_{t_{0}}+\int_{t_{0}}^{t}y_{s}ds, \\
y_{t} &=&y_{t_{0}}+\int_{t_{0}}^{t}a(x_{s},y_{s})ds+\sigma \left(
B_{t}-B_{t_{0}}\right) .
\end{eqnarray*}%
We start from the \Ito-Taylor expansion of $x:$
\begin{eqnarray*}
x_{t_{n+1}}
&=&x_{t_{n}}+hy_{t_{n}}+\int_{t_{n}}^{t_{n+1}}%
\int_{t_{n}}^{t}a(x_{s},y_{s})dsdt+\sigma I_{10}^{n+1}  \notag \\
&=&x_{t_{n}}+hy_{t_{n}}+ 0.5h^{2}a(x_{t_{n}},y_{t_{n}})+\sigma
I_{10}^{n+1}+O(h^{5/2}),  \label{IT2_x}
\end{eqnarray*}%
where $I_{10}^{n+1}:=\int_{t_{n}}^{t_{n+1}}\left( B_{t}-B_{t_{n}}\right)
dt$. To get higher order scheme for $y$, we apply \Ito's chain rule to
$a(x_{t},y_{t})$:
\begin{equation*}
a\left( x_{t},y_{t}\right) =a\left( x_{s},y_{s}\right) +\int_{s}^{t}[
a_{x}(x_{r},y_{r})y_{r}+( aa_{y}+0.5\sigma ^{2}a_{yy})
(x_{r},y_{r})] dr+\sigma \int_{s}^{t}a_{y}\left( x_{r},y_{r}\right)
dB_{r}.
\end{equation*}%
This leads to \Ito-Taylor expansion for $y$ (up to the order 2.0):%
\begin{eqnarray*}
y_{t_{n+1}} &=&y_{t_{n}}+\int_{t_{n}}^{t_{n+1}}a(x_{s},y_{s})ds+\sigma
\left( B_{t_{n+1}}-B_{t_{n}}\right)  \notag \\
&=&y_{t_{n}}+ha\left( x_{t_{n}},y_{t_{n}}\right) +\sigma \left(
B_{t_{n+1}}-B_{t_{n}}\right)  +a_{y}(x_{t_{n}},y_{t_{n}})\sigma
I_{10}^{n+1}+a_{yy}(x_{t_{n}},y_{t_{n}})\sigma ^{2}I_{110}^{n+1} \notag \\
&&+ 0.5h^{2}[ a_{x}(x_{t_{n}},y_{t_{n}})y_{t_{n}}+( aa_{y}+0.5
\sigma ^{2}a_{yy}) (x_{t_{n}},y_{t_{n}})]  +O\left(
h^{5/2}\right),  \notag  % \label{IT2_y}
\end{eqnarray*}%
where  $I_{110}^{n+1}=\int_{t_{n}}^{t_{n+1}}\int_{t_{n}}^{t}\left(B_{s}-B_{t_{n}}\right) dB_{s}dt$.
Representing $\sigma \left(B_{t_{n+1}}-B_{t_{n}}\right) $, $\sigma I_{10}^{n+1}$ and $I_{110}^{n+1}$ by $W_{n+1}$, $Z_{n+1}$ and $\frac{h}{6}(W_{n+1}^2-h)$ respectively, we obtain the scheme (\ref{IT2xy}).

\bibliographystyle{plain}
\bibliography{references}

\begin{thebibliography}{10}

\bibitem{And70}
E.~B. Andersen.
\newblock Asymptotic properties of conditional maximum-likelihood estimators.
\newblock {\em J. R. Stat. Soc. Series B}, pages 283--301, 1970.

\bibitem{AM11}
D.~F. Anderson and J.~C. Mattingly.
\newblock A weak trapezoidal method for a class of stochastic differential
  equations.
\newblock {\em Commun. Math. Sci.}, 9(1), 2011.

\bibitem{AI00}
L.~Arnold and P.~Imkeller.
\newblock The {K}ramers oscillator revisited.
\newblock In J.~Freund and T.~P\"oschel, editors, {\em Stochastic Processes in
  Physics, Chemistry, and Biology}, volume 557 of {\em Lecture Notes in
  Physics}, page 280. Springer, Berlin, 2000.

\bibitem{BD91}
P.~Brockwell and R.~Davis.
\newblock {\em Time series: theory and methods}.
\newblock Springer, New York, 2nd edition, 1991.

\bibitem{Bro01}
P.~J. Brockwell.
\newblock Continuous-time {ARMA} processes.
\newblock {\em Handbook of Statistics}, 19:249--276, 2001.

\bibitem{Bro14}
P.~J. Brockwell.
\newblock Recent results in the theory and applications of {CARMA} processes.
\newblock {\em Ann. Inst. Stat. Math.}, 66(4):647--685, 2014.

\bibitem{BDY07}
P.~J. Brockwell, R.~Davis, and Y.~Yang.
\newblock Continuous-time {G}aussian autoregression.
\newblock {\em Statistica Sinica}, 17(1):63, 2007.

\bibitem{CL15}
A.~J. Chorin and F.~Lu.
\newblock Discrete approach to stochastic parametrization and dimension
  reduction in nonlinear dynamics.
\newblock {\em Proc. Natl. Acad. Sci. USA}, 112(32):9804--9809, 2015.

\bibitem{DS04}
S.~Ditlevsen and M.~S{\o}rensen.
\newblock Inference for observations of integrated diffusion processes.
\newblock {\em Scand. J. Statist.}, 31(3):417--429, 2004.

\bibitem{FS01}
D.~Frenkel and B.~Smit.
\newblock {\em Understanding molecular simulation: from algorithms to
  applications}, volume~1.
\newblock Academic press, 2001.

\bibitem{Glo06}
A.~Gloter.
\newblock Parameter estimation for a discretely observed integrated diffusion
  process.
\newblock {\em Scand. J. Statist.}, 33(1):83--104, 2006.

\bibitem{GCF15}
G.~A. Gottwald, D.~Crommelin, and C.~Franzke.
\newblock Stochastic climate theory.
\newblock In {\em Nonlinear and Stochastic Climate Dynamics}. Cambridge
  University Press, 2015.

\bibitem{Ham94}
J.~D. Hamilton.
\newblock {\em Time Series Analysis}.
\newblock Princeton University Press, Princeton, NJ, 1994.

\bibitem{Hu96}
Y.~Hu.
\newblock Strong and weak order of time discretization schemes of stochastic
  differential equations.
\newblock In {\em S{\'e}minaire de Probabilit{\'e}s XXX}, pages 218--227.
  Springer, 1996.

\bibitem{hum05}
G.~Hummer.
\newblock Position-dependent diffusion coefficients and free energies from
  {Bayesian} analysis of equilibrium and replica molecular dynamics
  simulations.
\newblock {\em New J. Phys.}, 7(1):34, 2005.

\bibitem{Jen14}
A.~C. Jensen.
\newblock {\em Statistical Inference for Partially Observed Diffusion
  Processes}.
\newblock PhD thesis, University of Copenhagen, Faculty of Science, Department
  of Mathematical Sciences, 2014.

\bibitem{Jon81}
R.~H. Jones.
\newblock Jones fitting a continuous time autoregressive to discrete data.
\newblock In D.~F. Findley, editor, {\em Applied Time Series Analysis II},
  pages 651--682. Academic Press, New York, 1981.

\bibitem{KP99}
P.~E. Kloeden and E.~Platen.
\newblock {\em Numerical Solution of Stochastic Differential Equations}.
\newblock Springer, Berlin, 3rd edition, 1999.

\bibitem{KCG15}
D.~Kondrashov, M.~D. Chekroun, and M.~Ghil.
\newblock Data-driven non-{M}arkovian closure models.
\newblock {\em Physica D}, 297:33--55, 2015.

\bibitem{Kra40}
H.~A. Kramers.
\newblock Brownian motion in a field of force and the diffusion model of
  chemical reactions.
\newblock {\em Physica}, 7(4):284--304, 1940.

\bibitem{LLC15}
F.~Lu, K.~K. Lin, and A.~J. Chorin.
\newblock Data-based stochastic model reduction for the {Kuramoto--Sivashinsky}
  equation.
\newblock {\em arXiv:1509.09279}, 2015.

\bibitem{MH13}
A.~J. Majda and J.~Harlim.
\newblock Physics constrained nonlinear regression models for time series.
\newblock {\em Nonlinearity}, 26(1):201--217, 2013.

\bibitem{MSH02}
J.~C. Mattingly, A.~M. Stuart, and D.~J. Higham.
\newblock Ergodicity for {SDEs} and approximations: locally {L}ipschitz vector
  fields and degenerate noise.
\newblock {\em Stochastic Process. Appl.}, 101:185--232, 2002.

\bibitem{MST10}
J.~C. Mattingly, A.~M. Stuart, and M.~V. Tretyakov.
\newblock Convergence of numerical time-averaging and stationary measures via
  {Poisson} equations.
\newblock {\em SIAM J. Numer. Anal.}, 48(2):552--577, 2010.

\bibitem{MT07}
G.~N. Milstein and M.~V. Tretyakov.
\newblock Computing ergodic limits for {L}angevin equations.
\newblock {\em Physica D: Nonlinear Phenomena}, 229(1):81--95, 2007.

\bibitem{Nua06}
D.~Nualart.
\newblock {\em The Malliavin calculus and related topics}.
\newblock Springer-Verlag, 2nd edition, 2006.

\bibitem{Phi59}
A.~W. Phillips.
\newblock The estimation of parameters in systems of stochastic differential
  equations.
\newblock {\em Biometrika}, 46(1-2):67--76, 1959.

\bibitem{PSW09}
Y.~Pokern, A.~M. Stuart, and P.~Wiberg.
\newblock Parameter estimation for partially observed hypoelliptic diffusions.
\newblock {\em J. Roy. Statis. Soc. B}, 71(1):49--73, 2009.

\bibitem{Rao99}
P.~B.L.S. Rao.
\newblock {\em Statistical Inference for Diffusion Type Processes}.
\newblock Oxford University Press, 1999.

\bibitem{ST12}
A.~Samson and M.~Thieullen.
\newblock A contrast estimator for completely or partially observed
  hypoelliptic diffusion.
\newblock {\em Stochastic Process. Appl.}, 122(7):2521--2552, 2012.

\bibitem{SGH93}
L.~Schimansky-Geier and H.~Herzel.
\newblock Positive {L}yapunov exponents in the {K}ramers oscillator.
\newblock {\em J. Stat. Phys.}, 70(1-2):141--147, 1993.

\bibitem{Sor04}
H.~S{\o}rensen.
\newblock Parametric inference for diffusion processes observed at discrete
  points in time: a survey.
\newblock {\em Int. Stat. Rev.}, 72(3):337--354, 2004.

\bibitem{Sor12}
M.~S{\o}rensen.
\newblock Estimating functions for diffusion-type processes.
\newblock In M.~Kessler, A.~Lindner, and M.~S{\o}rensen, editors, {\em
  Statistical Methods for Stochastic Differential Equations}. Oxford University
  Press, London, 2012.

\bibitem{Tal02}
D.~Talay.
\newblock Stochastic {H}amiltonian systems: exponential convergence to the
  invariant measure, and discretization by the implicit {E}uler scheme.
\newblock {\em Markov Process. Related Fields}, 8(2):163 -- 198, 2002.

\end{thebibliography}
% \bibstyle{plain}

\end{document}